\UseRawInputEncoding
\documentclass{amsart}

\usepackage{amssymb,amsmath,amsthm}
\usepackage{amsaddr}
\usepackage{amsfonts}
\usepackage{authblk}
\usepackage{graphicx}
\usepackage{epstopdf}
\usepackage{ulem}
\usepackage{caption}
\usepackage{subcaption}
\usepackage{hyperref}
\usepackage{color}
\usepackage{marginnote}
\usepackage{float}
\usepackage{graphicx}
\usepackage{comment}
\usepackage{tikz-cd}
\usetikzlibrary{arrows}
\usetikzlibrary{intersections,positioning}
\tikzset{>=latex}
\usepackage{lipsum}%
\usepackage{a4wide}

\newcommand{\R}{\mathbb{R}}
\newcommand{\C}{\mathbb{C}}
\newcommand{\Z}{\mathbb{Z}}
\newcommand{\D}{\mathbb{D}}

\newcommand{\h}{\mathfrak{h}}

\newcommand{\authorfootnotes}{\renewcommand\thefootnote{\@fnsymbol\c@footnote}}%

\newcounter{newcounter}[section]
\numberwithin{equation}{section}
\numberwithin{newcounter}{section}
\numberwithin{figure}{section}
\numberwithin{footnote}{section}

\newtheorem{thm}[newcounter]{Theorem}
\newtheorem{defi}[newcounter]{Definition}
\newtheorem{prop}[newcounter]{Proposition}
\newtheorem{lem}[newcounter]{Lemma}
\newtheorem{cor}[newcounter]{Corollary}
\newtheorem{rem}[newcounter]{Remark}

\title[Birkhoff conjecture in Hill's lunar problem]{Birkhoff conjecture and finite energy foliations in Hill's lunar problem}
\date{\today}
\begin{document}
\maketitle

\email{liulei30@email.sdu.edu.cn}
\email{psalomao@sustech.edu.cn}

\begin{center}

\normalsize
\authorfootnotes
Lei Liu\textsuperscript{1} and
Pedro A. S. Salom\~ao\textsuperscript{2}  \par \bigskip

\textsuperscript{1}School of Mathematics, Shandong University  \par

\textsuperscript{2}Shenzhen International Center for Mathematics, SUSTech

\end{center}

\begin{abstract}
We prove Birkhoff's retrograde orbit conjecture in Hill's lunar problem by showing that the retrograde orbit bounds a disk-like global surface of section for every energy below the critical value. We also obtain a global description of the dynamics through the critical energy level by constructing finite energy foliations for energies slightly above it.  The binding of these foliations consists of the retrograde orbit together with the Lyapunov orbits near the critical points. As a consequence, there exist infinitely many periodic orbits and infinitely many trajectories asymptotic to the Lyapunov orbits. The proof combines pseudo-holomorphic curve techniques with a new convexity theorem for Hill's lunar problem. More precisely, we construct an explicit global symplectic change of coordinates under which the bounded regularized component becomes strictly convex up to the critical energy level. This convexity implies strong dynamical consequences, including lower bounds for the Conley-Zehnder indices of periodic orbits, and allows the application of the Hofer-Wysocki-Zehnder theory of finite energy foliations. As a result, we obtain disk-like global surfaces of section below the critical level and $2-3-2$  foliations for energies slightly above it.

\end{abstract}

\tableofcontents

\section{Introduction}

The restricted three-body problem has a central place in celestial mechanics since the works of Euler, Lagrange, Poincar¨¦, and Birkhoff. In its planar circular version, one studies the motion of a massless body under the gravitational influence of two primaries moving on circular orbits around their center of mass. Despite its simple formulation, the problem has a wide variety of dynamical phenomena, including periodic orbits, homoclinic intersections, chaotic scattering, and transport through the neck regions near the Lagrange equilibrium points.

One of the most influential models arising from the restricted three-body problem is Hill's lunar problem. Introduced by George W. Hill in his celebrated memoirs of 1878 and 1886 \cite{Hill1,Hill2}, the model was developed as part of Hill's program to obtain a precise analytical description of the motion of the Moon under the perturbative influence of the Sun. Rather than studying the full three-body problem directly, Hill derived a limiting system describing the dynamics near the Earth after a suitable rescaling procedure. The resulting equations contain the principal nonlinear effects of the solar perturbation while presenting a considerably simpler structure.

Hill's work represents one of the major achievements of nineteenth-century celestial mechanics. Using Fourier series and numerical calculations, he obtained highly accurate approximations of lunar periodic motions and set the foundations of a new lunar theory. His methods influenced generations of mathematicians and astronomers, and his lunar theory remained the standard reference
for many years. See Brown \cite{Brown1896} and Wintner \cite{Wintner1, Wintner2} for historical accounts and further developments. A modern discussion of Hill's original lunar theory, including the convergence of the associated series expansions, was given by Ligon \cite{Ligon}.


For much of the twentieth century, progress on Hill's lunar problem was driven primarily by perturbative methods, normal form techniques, and numerical investigations. Important advances were obtained through the study of invariant manifolds, periodic trajectories, and chaotic dynamics near the equilibrium points. In particular, the numerical investigations of Sim\'o and Stuchi \cite{SiSt2000} revealed a rich global structure involving Lyapunov orbits, stable and unstable manifolds, homoclinic and heteroclinic trajectories, and complicated transport phenomena near the critical energy level. The problem is known to be non-integrable \cite{MIW, M-RSS}, and transverse intersections of invariant manifolds produce chaotic recurrent dynamics in the sense of Poincar¨¦. See \cite{Conley1968,McGehee1969,SiSt2000} and references therein.

From the perspective of Hamiltonian and symplectic dynamics, Albers, Frauenfelder, Paternain and van Koert \cite{AFKP12} proved that the regularized energy surface below the critical value is fiberwise star-shaped. In particular, the regularized dynamics naturally defines a Reeb flow on a contact three-manifold. Later, Lee \cite{Lee2017} proved fiberwise convexity of the regularized energy surface below the critical value, relating the dynamics to a Finsler geodesic flow on the two-sphere. Quantitative aspects of the dynamics for very negative energies were studied in \cite{C-GLS2025} using Birkhoff normal forms and symplectic techniques. Also, recent works \cite{Aydin1, Aydin2} combined numerical investigations with techniques from symplectic dynamics to study bifurcating families of periodic solutions in the spatial Hill problem.

The development of pseudo-holomorphic curve techniques by Hofer \cite{Hofer1993} and Hofer, Wysocki and Zehnder \cite{ char1, char2, HWZI, HWZII, HWZIII, HWZ98,HWZ03} introduced powerful new methods for studying global properties of three-dimensional Hamiltonian dynamics. In particular, finite energy foliations organize the dynamics of Reeb flows through families of pseudo-holomorphic curves in symplectizations. These methods led to the construction of global surfaces of section and finite energy foliations in several situations, with strong dynamical consequences, see for example \cite{AFFHvK, dPS1, dPS2, dPHKS, FS, HWZ98, HWZ03,hryn2,HLS2014,HS1,HS2,HSW2023, LS25, Wendl08, Wendl08b, Wendl10}.

One of the central open questions in Hill's lunar problem concerns the existence of global surfaces of section. In 1915, Birkhoff \cite{Birkhoff1915} conjectured that the retrograde orbit bounds a disk-like global surface of section on the regularized energy surface below the critical value. In particular, it provides a complete reduction of the dynamics on the energy surface to a two-dimensional conservative system and implies the existence of a ``direct'' periodic orbit as a fixed point of the return map.

The present paper resolves Birkhoff's conjecture and gives the first global geometric description of Hill's lunar problem through the critical energy level. Our first main result proves that the retrograde orbit bounds a disk-like global surface of section for every energy below the critical value.
Our second main result concerns energies slightly above the critical value. In this case, we construct $2-3-2$ finite energy foliations whose binding consists of the retrograde orbit together with the Lyapunov orbits near the critical points. These foliations provide the first global description of the dynamics in a neighborhood of the critical energy level. In particular, they organize the dynamics by families of planes and cylinders transverse to the flow and imply the existence of infinitely many periodic orbits and infinitely many trajectories asymptotic to the Lyapunov orbits.

The results in this paper complete a program initiated by Birkhoff more than a century ago and further demonstrate the effectiveness of the Hofer-Wysocki-Zehnder theory of pseudo-holomorphic curves and finite energy foliations in the study of classical problems from celestial mechanics.

The key ingredient in the proofs is a convexity theorem for the bounded regularized component of the energy surface. More precisely, we show that after a suitable symplectic change of coordinates, the bounded component becomes strictly convex up to the critical energy level. This convexity implies lower bounds for the Conley-Zehnder indices of periodic orbits and provides the dynamical conditions required for the theory of pseudo-holomorphic curves. The resulting finite energy curves give rise to open book decompositions by disk-like global surfaces of section below the critical level, and $2-3-2$ foliations for energies slightly above it.

The convexity properties established in this paper appear to have applications beyond Hill's lunar problem. In particular, the symplectic coordinates constructed here naturally extend to the regularized planar restricted three-body problem for sufficiently small mass ratio \(\mu\). This suggests a possible approach to constructing finite energy foliations for energy levels ranging from the first to slightly above the second Lagrange value. Such foliations would provide a global description of the dynamics across the neck regions and will be investigated in future work.


\hfill\newline
\noindent{\bf Acknowledgments.}
LL is partially supported by the National Natural Science Foundation of China (Grant number: 12401238) and the Natural Science Foundation of Shandong Province (Grant number: ZR2024QA188).  LL thanks the support of the School of Mathematics at Shandong University. PS is partially supported by the National Natural Science Foundation of China (Grant number: w2431007). LL and PS thank the support of the Shenzhen International Center for Mathematics - SUSTech.

\section{Main Results}

The Hill's lunar problem can be described by the following Hamiltonian
$$
H(p,q)=\frac{1}{2}((p_1-q_2)^2+(p_2+q_1)^2)-\frac{3}{|q|}-\frac{3}{2}q_1^2.
$$
In suitable symplectic coordinates, the regularized Hamiltonian of Hill¡¯s lunar problem is given by
$$K(y,x) =
\frac{(y_1-2|x|^2x_2)^2}{2}+\frac{(y_2+2|x|^2x_1)^2}{2}+2|x|^2(2-3(x_1^2-x_2^2)^2).
$$
For every $E<0$,  $H^{-1}(E)$ corresponds to the regularized energy surface $K^{-1}(h), h=h(E):=12/|E|^{3/2}>0$, under a double covering map that identifies antipodal points. The quotient $K^{-1}(h(E))/\Z_2,$ will be called the regularization of $H^{-1}(E)$.

The Hamiltonian $H$ has a unique critical value $E_c=-9/2$, whose corresponding value of $K$ is $h_c=8\sqrt 2/9$.
For every $E<E_c$, there exists a component $\Sigma_E$ of $H^{-1}(E)$ projecting to a punctured disk about the origin in the $q$-plane. Let $\mathcal K_{h}$ be the corresponding regularized component of $K^{-1}(h), h=h(E)$. Then $\mathcal K_h$ is diffeomorphic to the three-sphere, and its quotient under the antipodal symmetry is denoted $\mathcal M_E \equiv \R P^3$ and called the regularization of $\Sigma_E$.
\begin{figure}[hbpt]
    \centering
    \begin{minipage}{0.3\linewidth}
    \includegraphics[width=\textwidth]{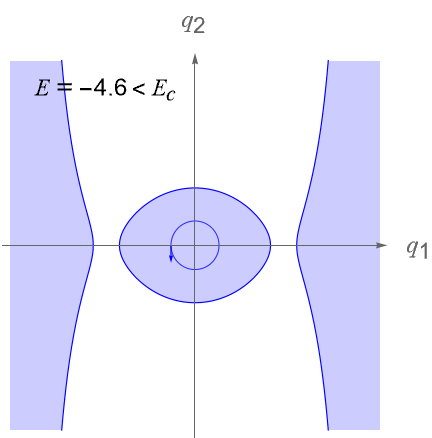}
    \end{minipage}
    \hspace{3mm}
    \begin{minipage}{0.3\linewidth}
    \includegraphics[width=\linewidth]{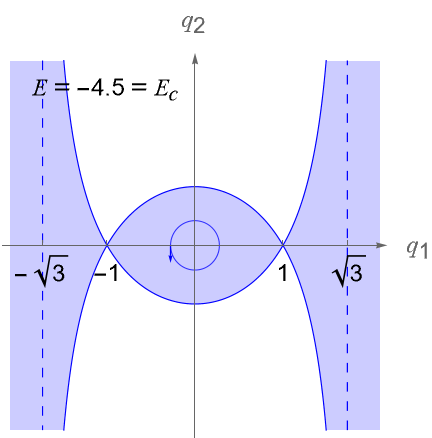}
    \end{minipage}
    \hspace{3mm}
    \begin{minipage}{0.3\linewidth}
    \includegraphics[width=\linewidth]{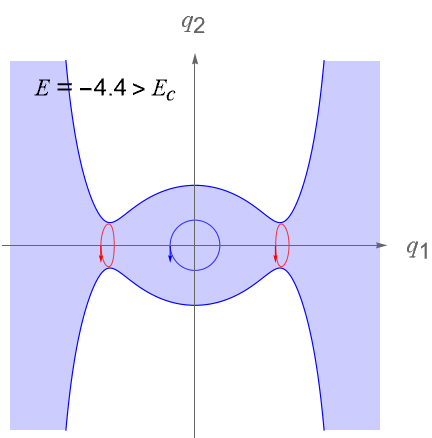}
    \end{minipage}
    \centering
    \caption{
Hill regions associated with the effective potential
$U(q)=-3/|q|-\frac32 q_1^2$
for energies below, at, and above the critical value
$E_c=-9/2$.
The bounded and unbounded Hill regions are disconnected when
$E<E_c$, meet at the saddle points $S_1$ and $S_2$ when
$E=E_c$, and become connected through two neck regions when
$E>E_c$. The retrograde orbit and the Lyapunov orbits
are shown.
}
    \label{fig_Hill}
\end{figure}

\subsection{Birkhoff conjecture and the dynamics below the critical value} In 1915, Birkhoff \cite[\S 16]{Birkhoff1915} introduced the shooting method to prove the existence of a retrograde orbit $P_{3, E}\subset \Sigma_E$ for every energy $E<E_c$, that is, a periodic orbit projecting to a simple closed curve in the $q$-plane surrounding the origin in the counterclockwise direction. Birkhoff's retrograde orbit conjecture states that the retrograde orbit bounds a disk-like global surface of section on the regularized energy surface $\mathcal M_E$, see \cite[Pg. 328]{Birkhoff1915}. More precisely, the retrograde orbit is $2$-unknotted in the following sense: there exists an immersion $u:\D \to \mathcal M_E$ from the unit closed disk $\D$ to $\mathcal M_E$ so that $u|_{\partial \D}:\partial \D \to P_{3, E}$ is a double covering map and $u|_{\D \setminus \partial \D}: \D \setminus \partial \D \to \mathcal M_E \setminus P_{3, E}$ is an embedding. Such an immersion $u$ is called a $2$-disk for $P_{3,E}$. Birkhoff's conjecture states that one can find a $2$-disk $\mathcal D_E\subset \mathcal M_E$ for $P_{3, E}$ whose interior is transverse to the flow, and every trajectory in $\mathcal M_E\setminus P_{3,E}$ intersects $\mathcal D_E$ forward and backward in time. Thus, a first return map to $\mathcal D_E\setminus P_{3,E}$ is well-defined and completely describes the dynamics on $\mathcal M_E$. As suggested by Birkhoff, a fixed point of the first return map, which always exists, would provide a ``direct'' orbit, whose existence does not follow directly from the shooting method.

Our first result completely proves Birkhoff's retrograde orbit conjecture in Hill's lunar problem for every $E<E_c$. Moreover, the resulting periodic orbit $\hat P$ given by any fixed point of the first return map forms a Hopf link with the retrograde orbit, and this link bounds an annulus-like global surface of section. This means that there exists an embedded closed annulus $\mathcal A_E\subset \mathcal M_E$ whose boundary is $P_{3, E} \cup \hat P$, whose interior is transverse to the flow, and so that every trajectory in $\mathcal M_E \setminus (P_{3,E}\cup \hat P)$ intersects $\mathcal A_E$ forward and backward in time. In this way, the first return map to the interior of $\mathcal A_E$ is well-defined and describes the dynamics on $\mathcal M_E$.

\begin{thm}\label{thm_main1}
For every $E<E_c$, let $P_{3, E}\subset \mathcal M_E\equiv \R P^3$ be the retrograde orbit constructed by Birkhoff in the regularization of $\Sigma_E$ whose projection to the $q$-plane is a simple closed curve surrounding the origin in the counterclockwise direction. Then the following statements hold:
\begin{itemize}
        \item[(i)] $P_{3,E}$ bounds a disk-like global surface of section $\mathcal D_E \subset \mathcal M_E$ which is the image of a $2$-disk for $P_{3,E}$. The interior of $\mathcal D_E$ is the page of an open book decomposition $(P_{3,E}, \Phi)$ of $\mathcal M_E$, i.e., there exists a smooth fibration $\Phi:\mathcal M_E \setminus P_{3,E} \to S^1$ so that the closure of each page $\Phi^{-1}(c)\subset \mathcal M_E \setminus P_{3,E},c\in S^1,$ is a $2$-disk for $P_{3,E}$ defining a global surface of section.

        \item[(ii)] Every fixed point of the first return map to $\mathcal D_E$ corresponds to a periodic orbit $\hat P\subset \mathcal M_E$ that forms a Hopf link with $P_{3,E}$, that is, $\hat P$ is $2$-unknotted and simply linked with $P_{3,E}$. Moreover, the Hopf link $P_{3,E} \cup \hat P$ bounds an annulus-like global surface of section for the flow on $\mathcal M_E$.  The interior of $\mathcal A_E$ is the page of an open book decomposition of $\mathcal M_E$ whose binding is $P_{3,E} \cup \hat P$ and whose pages are annulus-like global surfaces of section.
\end{itemize}
\end{thm}

\subsection{Finite energy foliations for energies slightly above the critical value} At the critical level $E=E_c=-9/2$, there exist two critical points $S_1=(0,1,-1,0)=-S_2$ of $H$. In fact, the lower level components $\Sigma_E \subset H^{-1}(E),$ develop the singularities $S_1,S_2$ as $E\to E_c^-$. In this case, the regularization of $H^{-1}(E_c)$ contains a compact subset homeomorphic to $\R P^3$, denoted $\mathcal M_{E_c}$, which, except for $S_1$ and $S_2$, contains only regular points.

The singularities $S_1,S_2 \in H^{-1}(E_c)$ project to saddle points of the associated potential $V(q):= -3/|q|-3q_1^2/2$. Each $S_j$ is a saddle-center for $H$, that is, the linearized Hamiltonian vector field at $S_j$ has a pair of real eigenvalues $\pm \sqrt{2\sqrt{7}+1}$ and a pair of purely imaginary eigenvalues $\pm i\sqrt{2\sqrt{7}-1}$.

For every $E-E_c>0$ sufficiently small, there exists a hyperbolic periodic orbit $P_{2,j, E}\subset H^{-1}(E)$ near each critical point $S_j$, and $P_{2,2,E} = - P_{2,1,E}$. These are the so-called Lyapunov orbits forming the center manifold of $S_j$ associated with the purely imaginary eigenvalues at $S_j$.  Also, for each $j$, it is possible to choose a regular embedded two-sphere $\mathcal S_j\subset H^{-1}(E)$ so that $P_{2,j, E}$ is an equator of $\mathcal S_j$ and each hemisphere of $\mathcal S_j \setminus P_{2,j, E}$ is transverse to the flow, pointing in opposite directions. Due to the antipodal symmetry of $H$, we may assume that $\mathcal S_2 = - \mathcal S_1$ and ${\rm dist}(\mathcal S_j, S_j) \to 0$ as $E \to E_c^+$. We keep the same notation for these objects after regularization. Indeed, there exists a compact subset $\mathcal M_E$ of the regularization of $H^{-1}(E)$, bounded by $\mathcal S_1 \cup \mathcal S_2$, which is diffeomorphic to a copy of $\R P^3$ with two copies of the open $3$-ball removed. This subset $\mathcal M_E$ is precisely where the relevant dynamics of Hill's lunar problem is located, and a rich structure of the dynamics follows from transverse homoclinic/heteroclinic orbits connecting the corresponding Lyapunov orbits.

For energies slightly above $E_c$, it is also possible to use Birkhoff's shooting method to obtain a retrograde orbit $P_{3, E} \subset \mathcal M_E\setminus (\mathcal S_1 \cup \mathcal S_2)$ satisfying the same properties as in the case of lower energies. It is one of our main goals in this paper to show that the retrograde orbit and the two Lyapunov orbits completely organize the flow on $\mathcal M_E$ despite its chaotic motion. To explain the meaning of organized flow, we need the following definition.

\begin{defi}[$2-3-2$ foliations] \label{def: 2-3-2 foliation}
For every $E-E_c>0$ sufficiently small, consider the compact subset $\mathcal M_E$ of the regularized $H^{-1}(E)$ that is bounded by $\mathcal S_1 \cup \mathcal S_2$. A $2-3-2$ foliation of $\mathcal{M}_E$, adapted to the dynamics of $\mathcal M_E$ and the orbit set $\mathcal{P}:=\{P_{3, E}, P_{2,1, E}, P_{2,2, E}$\}, is a singular foliation $\mathcal F$ of $\mathcal{M}_E$ whose singular set is given by $\cup_{P\in \mathcal P} P.$ The regular leaves of $\mathcal F$ are transverse to the flow and consist of the hemispheres of $\mathcal S_j\setminus P_{2,j, E}, j=1,2,$ two one-parameter families of planes asymptotic to $P_{3, E}^2$, and two rigid cylinders with a positive end at $P_{3, E}^2$ and a negative end at $P_{2,j, E},j=1,2$. At their ends, the families of planes break at a rigid cylinder and a hemisphere in $\partial \mathcal M_E$. The orbits in $\mathcal P$ are called the binding orbits of $\mathcal F$. See Figure \ref{fig: 2-3-2 foliation}.
\end{defi}

\begin{figure}[hbpt]
\centering  \includegraphics[width=.5\linewidth]{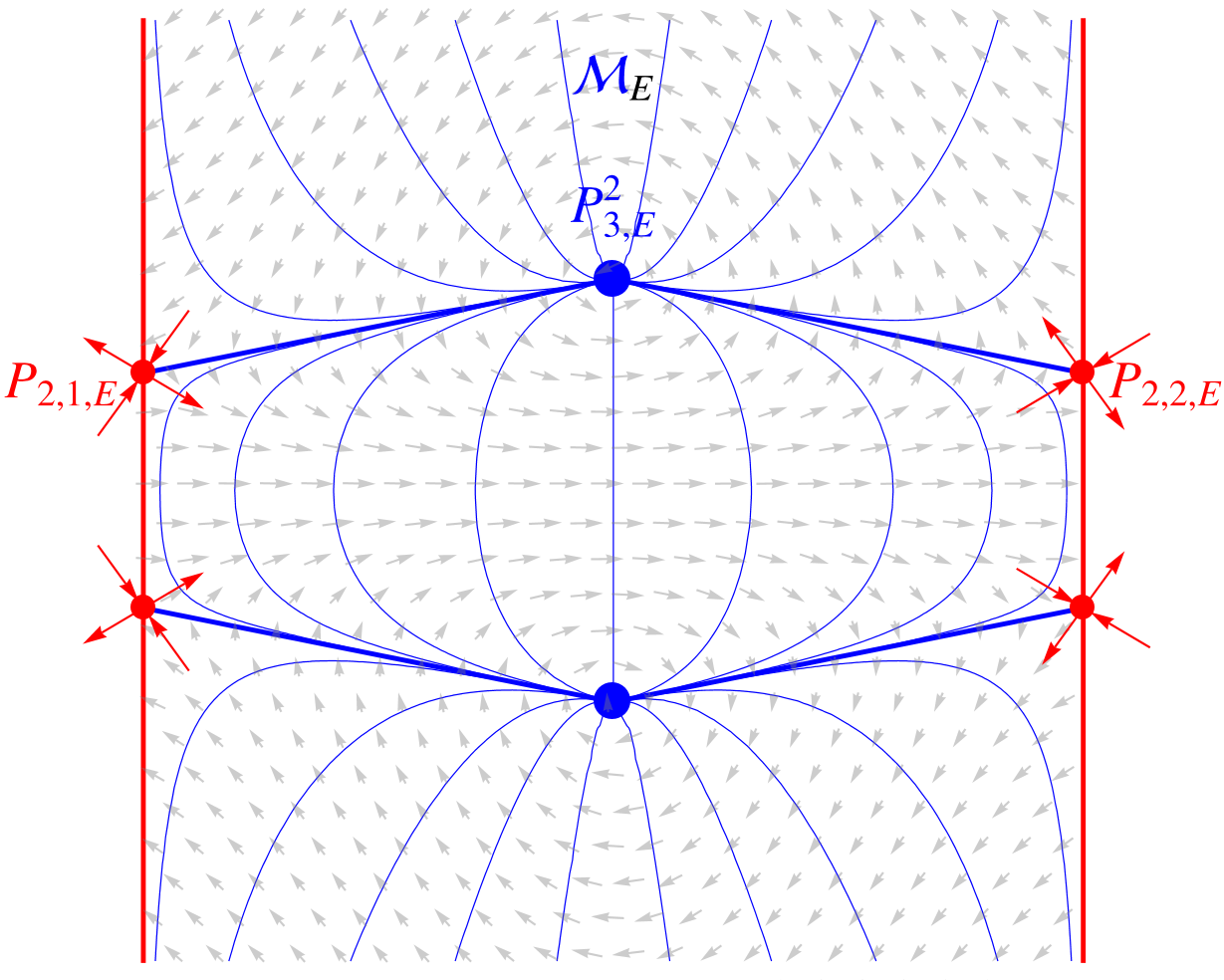}
\caption{The $2-3-2$ foliation on the regularized component $\mathcal{M}_E$ for energies slightly above the critical value $E_c$. The rigid cylinders (bold blue) have a positive end at the double cover of the retrograde orbit $P_{3,E}$ and a negative end at the Lyapunov orbit $P_{2,j,E}$ near the critical point $S_j$ for $j=1,2$. The union of the rigid planes asymptotic to $P_{2,j,E}$ (bold red) forms the two-spheres $\mathcal S_j\subset \partial \mathcal M_E$ for $j=1,2$.}
\label{fig: 2-3-2 foliation}
\end{figure}

Our second main result states that for $E-E_c>0$ sufficiently small, the regularized subset $\mathcal M_E$ admits a $2-3-2$ foliation. In particular, there exist infinitely many homoclinic or heteroclinic orbits in $\mathcal M_E$ connecting the Lyapunov orbits in $\partial \mathcal M_E$.

\begin{thm}\label{thm_main2}
For each $E-E_c>0$ sufficiently small, there exists a retrograde orbit $P_{3, E}\subset H^{-1}(E)$ and suitable regular two-spheres $\mathcal S_1, \mathcal S_2\subset H^{-1}(E)$ satisfying all properties above, so that the compact subset $\mathcal M_E$ of the regularized $H^{-1}(E)$ bounded by $\mathcal S_1\cup \mathcal S_2$ admits a $2-3-2$ foliation whose binding orbits are $P_{3,E}, P_{2,1,E}$ and $P_{2,2,E}$.

As a consequence of the $2-3-2$ foliation, $\mathcal M_E$ has infinitely many periodic orbits and infinitely many homoclinic orbits connecting the set of Lyapunov orbits $P_{2,1, E} \cup P_{2,2, E}$. Also, if the stable and unstable manifolds of $P_{2,1, E} \cup P_{2,2, E}$ do not coincide, then the flow on $\mathcal M_E$ has positive topological entropy.
\end{thm}

The foliations obtained in Theorems \ref{thm_main1} and \ref{thm_main2} arise as finite energy foliations in the sense of Hofer, Wysocki and Zehnder. In fact, the Hamiltonian vector field on $\mathcal M_E$ can be regarded as a Reeb vector field associated with a contact form $\alpha_E$ on $\mathcal M_E$. Its symplectization is the manifold $\R \times \mathcal M_E$ equipped with the symplectic form $d(e^a \alpha_E)$, where $a$ is the $\R$-coordinate. For suitable almost complex structures $J_E$ on $\R \times \mathcal M_E$, the leaves of the open book decompositions and the $2-3-2$ foliations obtained in those theorems are the projections to $\mathcal M_E$ of pseudo-holomorphic curves $\tilde u: (\dot \Sigma,j) \to (\R \times \mathcal M_E, J_E)$ defined on punctured Riemann surfaces $(\dot \Sigma,j)$ and satisfying the Cauchy-Riemann equation $d\tilde u \circ j = J_E(\tilde u) \circ d\tilde u$. It turns out that if such a curve has finite energy, then it behaves nicely near the punctures converging to periodic orbits of $\alpha_E$. Moreover, much can be said about their asymptotics, compactness properties, local deformations, algebraic invariants, etc. This theory was initiated by Hofer \cite{Hofer1993} in 1993 and was later fully developed by Hofer, Wysocki, and Zehnder \cite{char1, char2, HWZI, HWZII, HWZIII, HWZ98, HWZ03}. Theorems \ref{thm_main1} and \ref{thm_main2} strongly rely on the results built on this machinery.

\begin{rem}\label{rem: quotient}
Taking the quotient of $\mathcal M_E$ with respect to the extra antipodal symmetry of Hill's lunar problem in the original coordinates, one obtains $\overline{\mathcal M}_E = \mathcal M_E / \Z_2$ for every $E-E_c>0$ sufficiently small. Then $\overline{\mathcal M}_E$ is diffeomorphic to the lens space $L(4,3)$ with an open $3$-ball removed, whose boundary coincides with $\mathcal S_1 \sim \mathcal S_2$. In this case, the retrograde orbit $P_{3,E}$ projects to a $4$-unknotted periodic orbit $\overline{P}_{3,E}\subset \overline{\mathcal M}_E \setminus \partial \overline{\mathcal M}_E$. The proof of Theorem \ref{thm_main2} shows that the $2-3-2$ foliation $\mathcal F$ can be taken symmetric with respect to the antipodal symmetry and thus $\overline{\mathcal M}_E$ admits a $3-2$ foliation $\overline{\mathcal F}$ given by the natural $2$-to-$1$ projection of $\mathcal F$ to $\overline{\mathcal M}_E$. Its binding is $\overline{P}_{2, E}:=P_{2,1, E} \sim P_{2,2, E}$ and $\overline{P}_{3, E}$, and the regular leaves of $\overline{\mathcal F}$ consist of a family of planes asymptotic to the $4$-fold cover $\overline{P}_{3, E}^4$ of $\overline{P}_{3, E}$,  a rigid cylinder with a positive end at $\overline{P}_{3, E}^4$ and a negative end at $\overline{P}_{2, E}$, and a pair of rigid planes asymptotic to $\bar P_{2,E}$ projecting to the hemispheres of $\mathcal S_1\sim \mathcal S_2$.
\end{rem}

\subsection{Main steps of the proofs}
The proof of both Theorems \ref{thm_main1} and \ref{thm_main2} follows from a convexity property of the energy surfaces $\mathcal M_E, E\leq E_c,$ combined with results from \cite{HS2, HSW2023, LS25} establishing finite energy foliations with prescribed binding orbits.

We begin by explaining the geometric property.

\begin{defi}
Let $\mathcal M\hookrightarrow \R^4$ be an embedded topological three-sphere, possibly with a finite set $S$ of singularities, and smooth at the remaining points. We say that $\mathcal M$ is strictly convex if it bounds a convex subset of $\R^4$ and all the sectional curvatures of $\mathcal M\setminus S$ are positive. In particular, any hyperplane tangent to $\mathcal M$ at a regular point $p \in \mathcal M \setminus S$ has contact of order $1$ with $\mathcal M$ and intersects $\mathcal M$ only at $p$.
\end{defi}

The component $\mathcal M_E\simeq \R P^3$ of the regularized $H^{-1}(E), E\leq E_c,$ lifts to a topological three-sphere $\widetilde{\mathcal M}_E:=\mathcal K_{h(E)} \subset K^{-1}(h(E))$, and it is natural to ask whether $\widetilde{\mathcal M}_E$ is strictly convex. If this is the case, then many dynamical properties of the flow on $\mathcal M_E$ follow from the theory of pseudo-holomorphic curves, including, in particular, the Birkhoff conjecture.

However, numerical evidence shows that strict convexity does not always hold on $\widetilde{\mathcal M}_E$ for energies $E \leq E_c$ sufficiently close to $E_c$. In fact, the singular three-sphere $\widetilde{\mathcal M}_{E_c}$ has a conical behavior near its singularities $S_1, S_2 \in \widetilde{\mathcal M}_E$ and some sectional curvatures always approach zero at the singularities. Indeed, the sectional curvatures depend on higher-order terms of the Hamiltonian at the singularities, and in general, one cannot expect their positivity. This is the case of $\widetilde{\mathcal M}_E$, for $E\leq E_c$ sufficiently close to $E_c$.

Fortunately, we have some flexibility. We may consider symplectic diffeomorphisms $\Phi:(\R^4, \omega_0) \to (\R^4, \omega_0)$, $\omega_0 = \sum_{i=1,2} dy_i \wedge dx_i$, which preserve the Hamiltonian structure but may alter the convexity properties of the hypersurfaces.

Our first finding in this regard is the existence of suitable symplectic coordinates on the whole $\R^4$ so that the corresponding lifts $\widetilde{\mathcal M}_E$ are strictly convex for every $E \leq E_c$.

\begin{prop}\label{prop_strict_convexity}
There exists a smooth symplectic diffeomorphism $\Phi:\R^4 \to \R^4$  so that
$$
\widetilde {\mathcal M}_E':=\Phi(\widetilde{\mathcal M}_E) \mbox{ is strictly convex for every } E \leq E_c.
$$
More precisely, $\Phi$ is a symplectic shear, explicitly given in coordinates $(y,x)$ by
$$
\Phi(y,x) = (y +\nabla G_0(x),x), \qquad \forall (y,x) \in \R^4,
$$
where the generating function is $G_0(x):=\frac{2}{3}(x_1^2-x_2^2)x_1x_2$ for every $x\in \R^2.$ The new Hamiltonian becomes
\begin{equation}\label{new_Hamiltonian}
\begin{aligned}
\hat K(y,x):= K\circ \Phi\left(y,x \right) = \frac{(y_1-\frac{8}{3}x_2^3)^2}{2}+ \frac{(y_2+\frac{8}{3}x_1^3)^2}{2}+2(x_1^2+x_2^2)(2-3(x_1^2-x_2^2)^2).
\end{aligned}
\end{equation}
\end{prop}

The proof of Proposition \ref{prop_strict_convexity} is given in Section \ref{sec_convexity}. Strict convexity has been established for several Hamiltonian systems, see \cite{AFFHvK, dPS1, dPS2, GRS, HLOSY, JvK, Kim1, LS25, Salomao2004, Sch}. In many cases, it provides the geometric mechanism that connects concrete models in celestial mechanics with the theory of pseudo-holomorphic curves, finite energy foliations, and global surfaces of section.

We now explain how Proposition \ref{prop_strict_convexity} is used to prove Theorems \ref{thm_main1} and \ref{thm_main2}.
For energies $E\leq E_c$, one of the implications of Proposition \ref{prop_strict_convexity} is that $\mathcal M_E$ is {\bf dynamically convex}, that is, all of its contractible periodic orbits have Conley-Zehnder index at least $3$. This follows from well-known results in \cite{HWZ98} and \cite{Long02}.

\begin{cor}\label{cor_rp3}
    For every $E\leq E_c$, $\mathcal M_E$ is dynamically convex.
\end{cor}

Birkhoff's retrograde orbit conjecture in Hill's lunar problem, as stated in Theorem \ref{thm_main1}-(i), follows from a combination of Corollary \ref{cor_rp3} and the following theorem, which generalizes the results in \cite{char1,char2,hryn2,HS1} to Reeb flows on the universally tight $\R P^3$.

\begin{thm}[Hryniewicz-Salom\~ao {\cite[Corollary 1.8]{HS2}}]\label{thm_rp3}
Let $\alpha$ be a dynamically convex contact form on $\R P^3$. Let $P\subset \R P^3$ be a $2$-unknotted periodic orbit of $\alpha$ whose self-linking number is $-1/2$. Then $P$ binds an open book decomposition whose pages are $2$-disks for $P$, and each page is a global surface of section.
\end{thm}

For the definition of the (rational) self-linking number of a simple periodic orbit in $\R P^3$, see Section \ref{sec: Reeb dynamics and pseudo-holomorphic curves}. Global surfaces of section with more than one boundary component were considered in \cite{HSW2023}. The following statement is a particular case of the results in \cite{HSW2023} concerning the existence of annulus-like global surfaces of section bounded by a pair of periodic orbits forming a Hopf link.

\begin{thm}[Hryniewicz-Salom\~ao-Wysocki {\cite[Theorem 1.17]{HSW2023}}]\label{thm_annulus} Let $\alpha$ be a dynamically convex contact form on $\R P^3$. Let $P, \hat P\subset \R P^3$ be a pair of periodic orbits in $\R P^3$ forming a Hopf link, i.e., both $P$ and $\hat P$ are $2$-unknotted with self-linking number $-1/2$, and their lift to $S^3$ has linking number $+1$. Then the Hopf link $P\cup \hat P$ is the binding of an open book decomposition whose pages are annulus-like global surfaces of section.
\end{thm}

Combining Corollary \ref{cor_rp3} and Theorem \ref{thm_annulus} we obtain Theorem \ref{thm_main1}-(ii).

Now we consider energies $E$ slightly above $E_c$ and explain the strategy to prove Theorem \ref{thm_main2}. Again, the convexity of the regularized critical level proved in Proposition \ref{prop_strict_convexity} plays a central role in the proof of Theorem \ref{thm_main2}.

The first step is to construct a contact form $\alpha_E$ on a suitable subset $\mathcal M_E$ of the regularization of $H^{-1}(E)$, and a compatible almost complex structure in the symplectization of this subset, so that each Lyapunov orbit bounds a pair of pseudo-holomorphic planes approaching it through opposite directions.

We denote by $\mathcal M_{E_c}$ the sphere-like compact subset of the regularization of $H^{-1}(E_c)$ containing $S_1$ and $S_2$ and given as the limit of $\mathcal M_E$ as $E \to E_c^-$. In original coordinates, this is the subset of the critical level $H^{-1}(E_c)$ projecting to the $q$-plane as a punctured disk about the origin, containing the singularities $(\pm 1,0)$ at its boundary, and contained in the strip $-1\leq q_1 \leq 1$.

The following proposition establishes the desired contact geometric and analytical properties of the energy surfaces slightly above the critical level.

\begin{prop}\label{prop_contactform_J} For each $E-E_c>0$ sufficiently small, there exist a compact subset $\mathcal M_E$ of the regularization of $H^{-1}(E)$, a contact form $\alpha_E$ on $\mathcal M_E$ and an almost complex structure $J_E$ on $\R \times \mathcal M_E$ adapted to $\alpha_E$ so that the following conditions hold:
\begin{itemize}
    \item[(i)] The Reeb vector field $R_E$ of $\alpha_E$ is parallel to the Hamiltonian vector field restricted to $\mathcal M_E$.

    \item[(ii)] $\mathcal M_E$ is diffeomorphic to $\R P^3$ with two disjoint open $3$-balls removed. The boundary of $\mathcal M_E$ consists of two $2$-spheres $\mathcal S_j=\mathcal S_{j,E}, j=1,2,$ with $\mathcal S_2 = -\mathcal S_1$, so that  $\mathcal S_j$ contains the Lyapunov orbit $P_{2,j,E}$ and each hemisphere of $\mathcal S_j \setminus P_{2,j,E}$ is transverse to $R_E$ pointing in opposite directions. Moreover ${\rm dist}(\mathcal S_j, S_j) \to 0$ as $E \to E_c^+.$ The projection of $\mathcal M_E$ to the $q$-plane contains the corresponding projection of $\mathcal M_{E_c}$.

    \item[(iii)] For each $j=1,2,$ the hemispheres of $\mathcal S_j \setminus P_{2,j,E}$ are the projections to $\mathcal M_E$ of $J_E$-holomorphic planes asymptotic to $P_{2,j,E}$ through opposite directions.

    \item[(iv)] There exists a contact form $\alpha_{E_c}$ on  $\mathcal M_{E_c}$ such that the contact forms $\alpha_{E}$ in (ii) converge in $C^\infty_{\text{loc}}(\mathcal M_{E_c} \setminus (S_1 \cup S_2))$ to $\alpha_{E_c}$ as $E \to E_c^+$.

\end{itemize}
\end{prop}

In order to construct the $2-3-2$ foliations on $\mathcal M_E$ for $E-E_c>0$ sufficiently small, we need to show the existence of a retrograde orbit that will be used as a binding orbit. The following proposition shows that one can find a continuous family of such orbits for energies slightly above $E_c$.

\begin{prop}\label{prop_retrograde}
Let $\mathcal M_E$ and $\alpha_E$ be as in Proposition \ref{prop_contactform_J}.
There exists a continuous family of retrograde orbits $P_{3, E} \subset \mathcal{M}_{E}\setminus \partial \mathcal M_E$, defined on an open interval of energies $E$ containing $E_c$ so that the projection of $P_{3, E}$ to the $q$-plane is a simple closed curve circling the origin in the counterclockwise direction, and symmetric with respect to both axes. Moreover, $P_{3,E} \to P_{3,E_0}$ in $C^\infty$ as $E \to E_0$. In particular, the action $\mathcal A(P_{3,E}) = \int_{P_{3,E}}\alpha_E$ of $P_{3,E}$ is uniformly bounded near $E_c$, and there exists a continuous family of $2$-disks $\mathcal{D}=\mathcal{D}_{E}\subset \mathcal{M}_E$ for $P_{3,E}$ in $\mathcal M_E$ whose $|d\alpha_E|$-area $\mathcal S(\mathcal{D}, \alpha_E):=\int_{\mathcal D}|d\alpha_E|$ is uniformly bounded near $E_c$.
\end{prop}

In terms of contact topology, a retrograde orbit $P_{3,E}\subset \mathcal M_E$, as given in Proposition \ref{prop_retrograde} is not only $2$-unknotted but its self-linking number is $-1/2$ as a transverse knot on the contact manifold $(\mathcal M_E, \xi_E= \ker \alpha_E)$. This means that there exists a $2$-disk $u:\D \to \mathcal M_E\setminus \partial \mathcal M_E$ for $P_{3,E}$ so that the characteristic foliation $\mathcal C=(\ker (u^*\alpha_E)\cap  T\D)^\perp$ contains a unique singularity $e\in \D \setminus \partial \D$, which is nicely elliptic, and from which the regular leaves of $\mathcal C$ issue and intersect $\partial \D$ transversely. See \cite{HLS2014} for more details.

We also need some control on the indices of periodic orbits in the complement of the binding orbits, given by $P_{2,1, E}, P_{2,2, E}$, and $P_{3, E}$. We show that $\mathcal M_E$ contains no low-index orbit other than $P_{2,1,E}$ and $P_{2,2,E}$.
 This is a consequence of the strict convexity of the critical subset $\mathcal M_{E_c}$ proved in Proposition \ref{prop_strict_convexity}.

\begin{prop}\label{prop_weakly_convex}
Let $\mathcal M_E$ and $\alpha_E$ be as in Proposition \ref{prop_contactform_J}. Then for $E-E_c>0$ sufficiently small, the Conley-Zehnder index of every contractible periodic orbit in $\mathcal M_E \setminus \partial \mathcal M_E$ is at least $3$. In particular, $\mathcal M_E$ is weakly convex, i.e., the Conley-Zehnder index of its contractible periodic orbits is at least $2$, the only contractible periodic orbits with Conley-Zehnder index $2$ are the Lyapunov orbits $P_{2,1, E}, P_{2,2, E}\subset \partial \mathcal M_E$.
\end{prop}

To obtain a $2-3-2$ foliation for $\mathcal M_E$, we apply the following theorem from \cite{LS25}, which is a general criterion for the existence of finite energy foliations with prescribed binding orbits.

\begin{thm}[{Liu-Salom\~ao \cite[Theorem 1.3]{LS25}}]\label{thm: 2-3-2 foliation special}\label{thm_LS25}
Assume that $(\mathcal M,\alpha)$ is contactomorphic to the holed real projective space $\R P^3\setminus (\mathcal  B_1\cup\mathcal B_2)$ equipped with the universally tight contact structure. Here, $\mathcal B_1,\mathcal B_2\subset \R P^3$ are two disjoint and regular open $3$-balls. Denote by $\mathcal S_1,\mathcal S_2\simeq S^2$ the components of $\partial \mathcal M$. Let $J$ be a $d\alpha$-compatible almost complex structure in the symplectization $\R\times \mathcal M$ so that the following conditions hold:
\begin{itemize}
\item[(H1)] Each $\mathcal S_j\simeq S^2$ in $\partial \mathcal M$ admits an index-$2$ hyperbolic  periodic orbit $P_{2,j}\subset \mathcal S_j$ and the hemispheres of $\mathcal S_j\setminus P_{2,j}$ are the projections to $\mathcal M$ of embedded $J$-holomorphic planes asymptotic to $P_{2,j}$ through opposite directions. The Reeb vector field $R$ of $\alpha$ points in opposite directions along the hemispheres of $\mathcal S_j\setminus P_{2,j}$.

\item[(H2)] There exists a $2$-unknotted Reeb orbit $P_{3}\subset \mathcal M\setminus \partial \mathcal M$ with self-linking number $-1/2$ and rotation number $\rho(P_{3}^2)>1$.

\item[(H3)] The set of contractible periodic orbits $\mathcal P' \subset \mathcal M\setminus\{P_{3}\cup P_{2,1}\cup P_{2,2}\}$ satisfying
\begin{equation}\label{equ: linking conditions}
\rho(P')=1,\quad \mathrm{link}(P',P_{3}^2)=0\quad \text{and}\quad \mathcal A(P')\leq \mathcal S(\mathcal D,\alpha),
\end{equation} is empty, where $\mathcal S(\mathcal D, \alpha)=\int_{\mathcal D} |d\alpha|$ is the $|d\alpha|$-area of a $2$-disk $\mathcal D$ for $P_{3}$.
\end{itemize}
Then $\mathcal M$ admits a $2-3-2$ foliation $\mathcal F$ adapted to $\alpha$, $\{P_{3},P_{2,1},P_{2,2}\}$ and $J'$, where $J'$ is $C^\infty$-close to $J$ and coincides with $J$ near $\partial \mathcal M$.
\end{thm}

Hypotheses (H1), (H2) and (H3) in Theorem \ref{thm_LS25} follow from Propositions \ref{prop_contactform_J}, \ref{prop_retrograde} and \ref{prop_weakly_convex}, respectively. Indeed, Proposition \ref{prop_weakly_convex} implies that the Conley-Zehnder index of every contractible periodic orbit in $\mathcal M_E\setminus (P_{3, E} \cup P_{2,1, E} \cup P_{2,2, E})$ is $\geq 3$, which is equivalent to rotation number greater than $1$, thus implying that the set of contractible periodic orbits satisfying \eqref{equ: linking conditions} is empty. Hence, applying Theorem \ref{thm_LS25} to $\mathcal M_E$, $\alpha_E,$ $J_E$ and the binding orbits $P_{2,1,E}, P_{2,2,E}$ and $P_{3,E}$, we conclude that  $\mathcal M_E$ admits a $2-3-2$ foliation with binding orbits $P_{2,1,E}, P_{2,2,E}$ and $P_{3,E}$ for every $E-E_c>0$ sufficiently small.

The existence of infinitely many periodic orbits, homoclinic/heteroclinic orbits to the Lyapunov orbits, and positivity of topological entropy in case these stable-unstable manifolds do not coincide, follows from standard arguments found in \cite{Conley1968, HWZ03}, see also \cite{dPS2, dPHKS, LS25} for more details. This completes the proof of Theorem \ref{thm_main2}.

Sections \ref{sec_convexity}--\ref{sec_weakly_convex} are devoted to the proof of Propositions \ref{prop_strict_convexity}, \ref{prop_contactform_J},
\ref{prop_retrograde} and \ref{prop_weakly_convex}, respectively.

\section{Hill's lunar problem}\label{sec: contact form and retrograde orbit}

The circular planar restricted three-body problem describes the motion of a massless particle under the gravitational influence of two primaries moving on circular orbits around their center of mass. In rotating coordinates, its Hamiltonian is
$$
H_\mu(p,q)=\frac{(p_1-q_2)^2}{2}+\frac{(p_2+q_1)^2}{2}
-\frac{\mu}{|q-(1-\mu)|}-\frac{1-\mu}{|q+\mu|}-\frac{1}{2}|q|^2,
$$
where the primaries are located at $-\mu$ and $1-\mu$ in $\C\simeq\R^2$, with masses $1-\mu$ and $\mu$, respectively.

Since our interest is the dynamics near the smaller primary, it is convenient to shift the origin to the moon. After the change of coordinates
\[
(p,q)\mapsto (p-i(1-\mu),q+1-\mu),
\]
and the addition of the constant $(1-\mu)^2/2$, the Hamiltonian becomes
\begin{equation}\label{equ: hamiltonian RK}
\tilde H_\mu(p,q)
=
\frac12\big((p_1-q_2)^2+(p_2+q_1)^2\big)
-\frac{\mu}{|q|}
-\frac{|q|^2}{2}
-(1-\mu)\left(\frac1{|q+1|}+q_1\right).
\end{equation}
Hill introduced a further rescaling which magnifies a neighborhood of the moon as $\mu\to0$. Define
\[
\phi_\mu(p,q)
=
\big((\mu/3)^{1/3}p,(\mu/3)^{1/3}q\big)
\]
and
\[
\hat H_\mu
=
(\mu/3)^{-2/3}
\big(\tilde H_\mu\circ \phi_\mu+1-\mu\big).
\]
A direct computation gives
\[
\hat H_\mu(p,q)=\frac12\big((p_1-q_2)^2+ (p_2+q_1)^2\big)-\frac3{|q|}
-\frac{|q|^2}{2}-\frac{1-\mu}{(\mu/3)^{2/3}}\left(\frac{3^{1/3}}{|\mu^{1/3}q+3^{1/3}|}+\frac{\mu^{1/3}q_1}{3^{1/3}}-1\right).
\]
Expanding the expression above in $\mu>0$ small, we have
$$
\frac{3^{1/3}}{|\mu^{1/3}q+3^{1/3}|}
+\frac{\mu^{1/3}q_1}{3^{1/3}}-1 =\left(q_1^2-\frac{1}{2}q_2^2\right)(\mu/3)^{2/3}+O(\mu).
$$
Hence, passing to the limit $\mu\to0$, we obtain Hill's lunar Hamiltonian
\begin{equation}\label{equ: Hamiltonian Lunar}
\begin{aligned}
H(p,q)=\frac12\big((p_1-q_2)^2+(p_2+q_1)^2\big)-\frac3{|q|}
-\frac32 q_1^2 .
\end{aligned}
\end{equation}

Historically, Hill introduced this model in his study of lunar motion, leading to remarkably accurate predictions of the moon's periodic motion. The resulting system has since become one of the fundamental models in celestial mechanics.

The effective potential of \eqref{equ: Hamiltonian Lunar} is
\[
U(q):
=
-\frac3{|q|}
-\frac32 q_1^2 .
\]
Hamilton's equations are
\begin{equation}\label{equ: motion of H^0}
\left\{
\begin{aligned}
\dot q_1 &= p_1-q_2,\quad \dot p_1 = -\frac{3q_1}{|q|^3}+3q_1-(p_2+q_1),\\
\dot q_2 &= p_2+q_1,\quad \dot p_2 = -\frac{3q_2}{|q|^3}+p_1-q_2,
\end{aligned}
\right.
\end{equation}
or equivalently
\begin{equation}\label{equ: 2nd ODE}
\left\{
\begin{aligned}
\ddot q_1
&=
-2\dot q_2
-\frac{3q_1}{|q|^3}
+3q_1,\\
\ddot q_2
&=
2\dot q_1
-\frac{3q_2}{|q|^3}.
\end{aligned}
\right.
\end{equation}

The potential $U$ has exactly two critical points,
$
\bar S_1=(-1,0)$ and $
\bar S_2=(1,0),$
both lying on the critical level
$
U^{-1}(E_c),
E_c=-9/2.
$
These correspond to the equilibrium points
$
S_1=(0,1,-1,0),
$ and
$S_2=(0,-1,1,0)
$
of the Hamiltonian vector field. Moreover, the critical curve $U^{-1}(E_c)$ is contained in the strip
$
|q_1|<\sqrt3 ,
$
see Figure~\ref{fig_Hill}.

Now consider the Levi-Civita coordinates $q=v^2,p=u/(2\bar v)$, satisfying $dp\wedge dq=du\wedge dv$.
Denote the regularized Hamiltonian by $\hat K_0(u,v):=4|v|^2(H(u/(2\bar v),v^2)+c)$.
It depends on a parameter $c>0$ that can be eliminated under the rescaling $(u,v)=(c^{3/4}y,c^{1/4}x)$. We thus obtain
\begin{equation}\label{equ: Ham K}
\begin{aligned}
K(y,x):&=\frac{1}{c^{3/2}}(\hat K_0(c^{3/4}y,c^{1/4}x)+12)\\
&=\frac{(y_1-2|x|^2x_2)^2}{2}+\frac{(y_2+2|x|^2x_1)^2}{2}+V(x),
\end{aligned}
\end{equation}
where the magnetic field is $F=(-2|x|^2x_2,2|x|^2x_1)$ and the potential function is
$$
V(x)=-6|x|^6+24|x|^2x_1^2x_2^2+4|x|^2=2|x|^2(2-3(x_1^2-x_2^2)^2).
$$
Let $h=h(E)=12/|E|^{3/2}$. For every $E<0$, the energy level $K^{-1}(h)$ double covers
the regularization of the energy surface $H^{-1}(E)$. In particular, the critical level $H^{-1}(E_c)$ is associated with the critical level $K^{-1}(h_c)$ with critical value $h_c=8 \sqrt{2}/9$.

\section{Basics on finite energy foliations}\label{sec: Reeb dynamics and pseudo-holomorphic curves}

We briefly recall the notions from contact geometry and pseudo-holomorphic curve theory that will be used later.

Consider the standard contact three-sphere $(S^3,\xi_0)$, where
$
\xi_0=\ker \alpha_0,
$
and $\alpha_0$ is the restriction to $S^3\subset \mathbb R^4$ of the Liouville form
\[
\alpha_0=\frac12\sum_{j=1}^2(y_jdx_j-x_jdy_j).
\]
Here,
$
S^3=\left\{(x_1,y_1,x_2,y_2)\in \mathbb R^4 \mid x_1^2+y_1^2+x_2^2+y_2^2=1\right\}.
$

If $f:S^3\to \mathbb R^+$ is smooth and positive, then
$
\alpha=f\alpha_0
$
is again a contact form defining the same contact structure $\xi_0$. Its Reeb vector field $R$ is characterized by
$\iota_R d\alpha =0$ and $\alpha(R)=1.$
The associated flow will be denoted by
$\psi_t:S^3\to S^3.$ Since the Reeb flow preserves $\alpha$, it also preserves the contact distribution $\xi_0$.

A periodic Reeb orbit is a pair $P=(x,T)$, where $T>0$ and
$x:\mathbb R\to S^3$ satisfies
$\dot x(t)=R(x(t))$, $x(t+T)=x(t).$
The orbit is called nondegenerate when the linearized return map
$d\psi_T(x(0)):\xi_0|_{x(0)}\to \xi_0|_{x(0)}$ has no eigenvalue equal to $1$.

The standard contact form $\alpha_0$ is invariant under the $\mathbb Z_p$-action generated by
\[
(z_1,z_2)\mapsto
\left(
e^{2\pi i/p}z_1,
e^{2\pi iq/p}z_2
\right),
\]
where $\mathbb C^2$ is identified with $\mathbb R^4$ through
$
z_j=x_j+iy_j.
$
Therefore, $\alpha_0$ descends to the lens space
\[
L(p,q)=S^3/\mathbb Z_p.
\]
We denote by
$
\pi_{p,q}:S^3\to L(p,q)
$
the quotient map. The induced contact structure on $L(p,q)$ is still denoted by $\xi_0$, and is known as the universally tight contact structure.

A knot
$
K\subset L(p,q)
$
is said to be $p$-unknotted if there exists an immersion
$
u:\mathbb D\to L(p,q)
$
such that
$
u|_{\mathbb D\setminus \partial\mathbb D}
$
is an embedding into $L(p,q)\setminus K$, while the boundary map
$
u|_{\partial\mathbb D}:\partial\mathbb D\to K
$
is a covering of degree $p$. Such a map is called a $p$-disk for $K$.

Suppose now that $K$ is transverse to $\xi_0$ and admits a $p$-disk $u$. Choose a non-vanishing section $X\in \Gamma(u^*\xi_0)$ defined near the boundary. Pushing a $p$-fold cover of $K$ slightly in the direction of $X$ produces a knot $K'$ disjoint from $K$. The rational self-linking number is defined by
\[
{\rm sl}(K)=\frac{K'\cdot u}{p^2},
\]
where $K'\cdot u$ denotes the algebraic intersection number. Orientations are chosen as follows: $K$ is oriented by the Reeb flow, $K'$ inherits the same orientation, the disk $u$ is oriented compatibly with $K$, and the ambient manifold is oriented by $\alpha_0\wedge d\alpha_0>0.$

A basic example is obtained from
$K=\pi_{p,q}(S^1\times\{0\})\subset L(p,q).$ A $p$-disk is given by
$\pi_{p,q}\circ u,$ where $
u(z)=(z,\sqrt{1-|z|^2}), z\in \mathbb D.$ A direct computation shows that ${\rm sl}(K)=-1/p.$ The value $-1/p$ is distinguished. In fact, a transverse $p$-unknot in $(L(p,q),\xi_0)$ has self-linking number equal to $-1/p$ if and only if one can choose a $p$-disk whose characteristic foliation has exactly one singular point, and this singularity is elliptic in the standard sense. More precisely, if $V$ locally defines the characteristic foliation, then the linearization $DV$ at the singular point has two real eigenvalues with the same sign.

We next recall the definition of the rotation number. Let $P=(x,T)$ be a contractible periodic Reeb orbit in $(L(p,q),\xi_0)$. Hence $P$ can be seen as a periodic orbit in $S^3$. Denote by $\varphi_t$ the Reeb flow and by $d\varphi_t:\xi_0|_{x(0)}\to \xi_0|_{x(t)}$ its linearization along the orbit. Choose a global framing $\{V_1,V_2,V_3\}$
of the tangent bundle of $S^3$ such that $V_3=R$ and $(\alpha_0\wedge d\alpha_0)(V_1,V_2,R)>0.$ Writing
\[
d\varphi_t(V_1(x(0)))=a_1(t)V_1(x(t))+a_2(t)V_2(x(t))+a_3(t)R(x(t)),
\]
we observe that $(a_1(t),a_2(t))\neq (0,0)$ for every $t$. Hence, the complex-valued function $a_1(t)+ia_2(t)$ admits a continuous argument $\theta(t),
\theta(0)=0.$ The rotation number of $P$ is defined by
\begin{equation}\label{rho}
\rho(P)=T\lim_{t\to\infty}\frac{\theta(t)}{2\pi t}.
\end{equation}
The limit exists and is independent of all auxiliary choices.

We now turn to pseudo-holomorphic curves. Let $(\mathbb R\times Y,J)$ be the symplectization of a contact three-manifold $(Y,\alpha)$. Write $a$ for the $\mathbb R$-coordinate. An almost complex structure $J$ is said to be compatible with $\alpha$ if
\[
J\cdot \partial_a=R, \qquad
J(\xi_0)=\xi_0,
\]
and the bilinear form $d\alpha(\cdot,J\cdot)$ defines a positive-definite metric on $\xi_0$. The space of such almost complex structures will be denoted by $\mathcal J(\alpha).$

Let $(\Sigma,j)$ be a connected Riemann surface, possibly with boundary, and let
$\Gamma\subset \Sigma\setminus \partial\Sigma$ be a finite set of punctures. Define
$\dot\Sigma=\Sigma\setminus \Gamma.$
A map $\tilde u=(a,u):\dot\Sigma\to \mathbb R\times Y$ is called a finite energy $J$-holomorphic curve if it satisfies the Cauchy-Riemann equation
\[
d\tilde u\circ j=J(\tilde u)\circ d\tilde u,
\]
or equivalently $\bar\partial_J\tilde u=0,$ and has finite Hofer energy
$0<E(\tilde u)<\infty,$ where
\[
E(\tilde u)=\sup_{\phi\in \mathcal T}
\int_{\dot\Sigma}
\tilde u^*d(\phi(a)\alpha), \qquad
\mathcal T=\left\{
\phi:\mathbb R\to [0,1]\mid\phi'\ge 0
\right\}.
\]

Finite energy curves were introduced by Hofer \cite{Hofer1993} in order to study periodic orbits of Reeb flows through holomorphic curve methods. In Reeb dynamics of dimension three, they have proved particularly useful because of their very rigid asymptotic behavior.

The curve $\tilde u$ is called nicely embedded when the projected map $u:\dot\Sigma\to Y$ is an embedding. Every non-removable puncture
$z_0\in \Gamma$ has a sign
$\epsilon(z_0)\in\{-1,+1\},$
determined by the behavior of the $\mathbb R$-coordinate:
$a(z)\to \epsilon(z_0)\infty$
as $z\to z_0.$ If $\epsilon(z_0)=+1$, the puncture is called positive; otherwise, it is called negative.
Choose holomorphic polar coordinates
$(s,t)\in [0,\infty)\times \mathbb R/\mathbb Z$ near a puncture. Then, after passing to a subsequence if necessary, there exists a periodic Reeb orbit $P=(x,T)$ such that
$u(s_n,\cdot)\to x(\epsilon(z_0)T\cdot)$ in $C^\infty$ as $s_n\to\infty$. The orbit $P$ is called an asymptotic limit of the curve at the puncture.

When the asymptotic orbit is nondegenerate, the asymptotic limit is unique, and the convergence admits a precise asymptotic description. This fact plays a central role in compactness and intersection arguments for pseudo-holomorphic curves.

\subsection{Weakly convex foliations}\label{sec: weakly convex foliations}

Consider the $l$-holed lens space
$
\mathcal M=L(p,q)\setminus \bigcup_{j=1}^l \mathcal B_j,
$
equipped with the standard contact structure, where the $\mathcal B_j\subset L(p,q)$ are pairwise disjoint open $3$-balls. Then
\[
\partial \mathcal M=\bigcup_{j=1}^l \mathcal S_j,
\qquad
\mathcal S_j=\partial \mathcal B_j,
\]
is a disjoint union of embedded two-spheres. We continue to denote by $\alpha_0$ and $\xi_0$ the restrictions of the standard contact form and contact structure to $\mathcal M$.

Fix a contact form $\alpha=f\alpha_0$ on $\mathcal M$, where $f:\mathcal M\to \mathbb R^+$ is smooth. Assume that each boundary component $\mathcal S_j$ contains a hyperbolic periodic Reeb orbit
$
P_{2,j}\subset \mathcal S_j
$
which divides $\mathcal S_j$ into two hemispheres. We assume that the Reeb vector field points in opposite directions along the two hemispheres of
$
\mathcal S_j\setminus P_{2,j}.
$
In particular, the orbit $P_{2,j}$ has rotation number equal to $1$. Assume that the hemispheres of $\mathcal S_j \setminus P_{2,j}$ are projections of $J$-holomorphic planes, where $J\in \mathcal J(\alpha)$.

Assume moreover that the interior of $\mathcal M$ contains a $p$-unknotted periodic orbit
$
P_3\subset \mathcal M\setminus \partial \mathcal M
$
with self-linking number
$
{\rm sl}(P_3)=-1/p
$
and rotation number
$
\rho(P_3^p)>1.
$
We denote the collection of distinguished periodic orbits by
$
\mathcal P=\{P_3,P_{2,1},\ldots,P_{2,l}\},$ and call it the binding.

Finite energy foliations adapted to these data were introduced in several related forms in the literature on pseudo-holomorphic curves and Reeb dynamics, especially in the work of Hofer, Wysocki and Zehnder, and later developments in \cite{dPS1, dPHKS, FS, char1,char2,HWZI, HWZII, HWZIII, HWZ98, HWZ03, Wendl08, Wendl08b, Wendl10a, Wendl10}. Roughly speaking, such foliations organize the global dynamics by means of embedded holomorphic curves asymptotic to prescribed Reeb orbits.

\begin{defi}\label{defi: Finite energy foliation}
A finite-energy foliation adapted to $\alpha$, $\mathcal P$ and $J$ is a regular foliation $\tilde{\mathcal F}$
of the symplectization $\mathbb R\times \mathcal M$ whose leaves are either:
\begin{itemize}
\item trivial cylinders over the binding orbits in $\mathcal P$, or
\item images of nicely embedded finite-energy $J$-holomorphic curves with uniformly bounded energy and asymptotic limits contained in $\mathcal P$.
\end{itemize}
Moreover, if $F\in \tilde{\mathcal F},$
then every $\mathbb R$-translate $F+a, a\in \mathbb R,$ also belongs to $\tilde{\mathcal F}$.
\end{defi}

Projecting the leaves to $\mathcal M$ produces a singular foliation transverse to the Reeb flow away from the binding orbits. This leads to the following notion.

\begin{defi}[\cite{dPS3}]\label{def: weakly convex foliation}
A weakly convex foliation $\mathcal F$ of $\mathcal M$, adapted to $\alpha$ and $\mathcal P$, is a singular foliation whose singular set is $\bigcup_{P\in \mathcal P} P,$ and such that the complement of the singular set is foliated by properly embedded surfaces transverse to the Reeb vector field. Each regular leaf
$\dot\Sigma\hookrightarrow \mathcal M\setminus \bigcup_{P\in\mathcal P}P$
is a punctured sphere $\dot\Sigma=\mathbb CP^1\setminus \Gamma,\, 0<\#\Gamma<\infty,$ whose punctures are asymptotic to some orbit $P_{2,j}$ or to $P_3^p$. The closure $\Sigma\subset \mathcal M$ is required to be an immersion. A puncture $z\in \Gamma$ is called positive if the boundary orientation induced on the asymptotic orbit agrees with the Reeb flow orientation. Otherwise, it is called negative. The asymptotic limits of a leaf are all distinct, and every leaf has exactly one positive puncture. If a leaf is asymptotic to $P_3^p$, then this puncture is positive and all remaining punctures are negative. If $P_3^p$ is not an asymptotic limit, then the leaf is simply one of the hemispheres of $\mathcal S_j\setminus P_{2,j}$ for some $j$.
\end{defi}

For generic almost complex structures, the leaves appearing in a weakly convex foliation are particularly simple. They consist only of:
\begin{itemize}
\item planes asymptotic to $P_3^p$,
\item cylinders with one positive puncture at $P_3^p$ and one negative puncture at some $P_{2,j}$,
\item planes projecting onto the hemispheres of $\mathcal S_j\setminus P_{2,j}$.
\end{itemize}

More precisely, there are exactly $l$ one-parameter families of planes asymptotic to $P_3^p$, together with $l$ rigid cylinders connecting $P_3^p$ to the orbits $P_{2,j}$. Each family of planes degenerates at its ends into a building formed by one rigid cylinder and one boundary hemisphere. Thus the rigid cylinders separate the different families of planes.

When $l=1$, the resulting configuration is called a $3-2$ foliation. When $l=2$, it is called a $2-3-2$ foliation. These names reflect the indices of the binding orbits involved in the construction.

We now state the hypotheses under which such foliations exist.

\begin{thm}[{\cite[Theorem 1.3]{LS25}}]\label{thm_weakly_convex}
Let $\mathcal M$, $\alpha$ and $J$ satisfy the conditions above. Let $\mathcal D\subset \mathcal M \setminus (P_3\cup \cup_{j=1}^l P_{2,j})$ be a $p$-disk for $P_3$, and let $\mathcal S(P_3,\alpha):=\int_{\mathcal D} |d\alpha|$ be its $d\alpha$-area. Let $\mathcal P'$
denote the set of contractible periodic Reeb orbits $P'$ in
$\mathcal M\setminus \left(P_3\cup \cup_{j=1}^l P_{2,j}\right)$ satisfying
$\rho(P')=1,\, {\rm link}(P',P_3^p)=0,$
and $\mathcal A(P')\le \mathcal S(\mathcal D).$ If $\mathcal P'=\emptyset,$ then $\mathcal M$ admits a weakly convex foliation adapted to $\alpha$, $\{P_3,P_{2,1},\ldots,P_{2,l}\}$ and $J'$, where $J'\in \mathcal J(\alpha)$ is $C^\infty$-arbitrarily close to $J$ and coincides with $J$ near $\partial \mathcal M$.
The leaves consist of:
\begin{itemize}
\item $l$ families of planes asymptotic to $P_3^p$,
\item $l$ rigid cylinders connecting $P_3^p$ to the orbits $P_{2,j}$,
\item for each $j$, two boundary planes asymptotic to $P_{2,j}$ projecting onto the hemispheres of $\mathcal S_j\setminus P_{2,j}$.
\end{itemize}
In particular, when $l=1$ one obtains a $3-2$ foliation, while for $l=2$ one obtains a $2-3-2$ foliation.
\end{thm}
One of the consequences of a weakly convex foliation is the existence of homoclinic or heteroclinic trajectories to some Lyapunov orbits.

\begin{cor}[{\cite[Theorem~1.4]{LS25}}]\label{cor: homoclinic}
Assume that the orbit $P_{2,j_0}$ has the largest action among the boundary Lyapunov orbits appearing in Theorem~\ref{thm_weakly_convex}. Then either:
\begin{itemize}
\item $P_{2,j_0}$ possesses a homoclinic orbit, or
\item there exists a heteroclinic orbit from $P_{2,j_0}$ to another orbit $P_{2,j'},j'\neq j_0.$
\end{itemize}
In particular, when $l=1$, a homoclinic orbit to $P_{2,1}$ always exists.
\end{cor}

\section{Proof of Proposition \ref{prop_strict_convexity}}
\label{sec_convexity}

In this section, we study the convexity properties of the Hamiltonian \eqref{new_Hamiltonian} given in Proposition \ref{prop_strict_convexity}.  For simplicity we consider instead of $\hat K$, the rescaled Hamiltonian $\frac{9}{64}\hat K(\frac{8}{3}y,x)$, which is still denoted by $\hat K$ and takes the form
\begin{equation}\label{hatK}
\hat K(y,x) =\frac{(y_1-x_2^3)^2}{2}+ \frac{(y_2+x_1^3)^2}{2}+\frac{9}{32}(x_1^2+x_2^2)(2-3(x_1^2-x_2^2)^2).
\end{equation}
In this case, the critical level is $\mathfrak h_c =1/(4\sqrt{2})$. Notice that the convexity properties are unchanged under the rescalings of $\hat {K} $.

We have to show that for every energy $h\leq \mathfrak h_c$, the corresponding sphere-like regularized component $\mathcal M_h= \hat K^{-1}(h)$ is strictly convex.  To do that, we shall prove that the tangential Hessian of $\hat K$ is positive-definite at every regular point of $\mathcal M_h\setminus S,$ where $S$ is the singular set of $\mathcal M_h$. Recall that $S=\emptyset$ for every $E <\mathfrak h_c$ and $S$ contains $4$ singular points for $h=\mathfrak h_c$. It turns out that this local convexity condition implies the global strict convexity property for $\mathcal M_h$, see \cite{LS25}, since all four singularities are located in a common hyperplane.

To prove the local convexity condition, we use a criterion introduced in \cite{LS25} for the class of mechanical-magnetic Hamiltonians of the form
\begin{equation}\label{equ: magnetic hamiltonian}
\mathrm K(y,x)=\frac{(y_1+f_1(x))^2}{2}+\frac{(y_2+f_2(x))^2}{2}+V(x),\quad \forall y,x\in \mathbb{R}^2,
\end{equation}
where $F(x)=(f_1(x),f_2(x))$ is a smooth magnetic field and $V(x)$ is a smooth potential function.

Fix a sphere-like subset $\mathcal{M}\subset \mathrm K^{-1}(h)$ and assume that $\mathcal{M}$ has at most finitely many singularities of saddle-center type. Denote by $S$ the singular set of $\mathcal{M}$. Consider the projection $\mathcal H \subset \R^2$ of $\mathcal M$ under
$$
\Pi_x:\R^4\to \R^2, \qquad (y,x) \mapsto x.
$$

We assume that $\mathcal H=\Pi_x(\mathcal M)$ is a topological closed disk whose interior $\dot{\mathcal{H}}$ corresponds to points $(y,x) \in \mathcal M \setminus S$ satisfying $V(x) <h$ and thus $y\neq-F(x)$. The points projecting to the boundary $\partial \mathcal{H}$ satisfy $V(x) =h$ and thus $y=-F(x)$.  

The following criterion is proved in \cite{LS25}.

\begin{thm}[{\cite[Theorem 9.1]{LS25}}] \label{thm: convexity formula}
Assume that $F(x)=(f_1(x),f_2(x))$ is decoupled, i.e., the mixed second order derivative $f_{j,12}$ of $f_j, j=1,2,$ vanish identically. Denote $(s,t):=(\cos \theta, \sin \theta)$. Then the sectional curvatures at $(y,x)\in \mathcal M\setminus S$ are all positive if and only if
\begin{equation}\label{G}
\begin{aligned}
G(\theta,x):= &r^2c(\theta,x)d(\theta,x)+c(\theta,x)V_2(x)^2+d(\theta,x)V_1(x)^2\\ & -r^2V_{12}(x)^2-2V_{12}(x)V_1(x)V_2(x)>0,
\end{aligned}
\end{equation}
for every $\theta\in \R / 2\pi \Z, x\in \mathcal H \setminus \Pi_x(S)$, where $r=r(x)>0$ and $\theta\in \R / 2\pi \Z$ are determined by $y+F(x)=r(x)e^{i\theta}=\sqrt{2(h-V(x))}e^{i\theta}$, and
$$
\begin{aligned}
    c(\theta, x) &=r(sf_{1,11}(x)+tf_{2,11}(x))+V_{11}(x), \\
    d(\theta,x) & =r(sf_{1,22}(x)+tf_{2,22}(x))+V_{22}(x).
\end{aligned}
$$
In particular, if $G(\theta,x)>0$ for every $(\theta,x)\in \R / 2\pi \Z \times (\mathcal H\setminus \Pi_x(S))$, then $\mathcal M$ is strictly convex.
Here, we use the notation $V_j = \partial_{x_j}V$, $V_{ij}=\partial^2_{x_jx_i}V$, etc.
\end{thm}

Notice that if $r(x)=0$, then $G(\theta,x)$ does not depend on $\theta$. Also, if  $F\equiv 0$  (that is, ${\rm K}$ is a mechanical Hamiltonian), then this criterion coincides with the one considered in \cite{Salomao2004}.

The magnetic field of the regularized Hamiltonian $\hat K$ given in \eqref{hatK} is decoupled and given by
$$
\hat F(x):=(\hat f_1(x_2),\hat f_2(x_1))=(-x_2^3,x_1^3).
$$
The new potential writes as
$$
\hat V(x):=\frac{9}{32} (x_1^2 + x_2^2) (2 - 3 (x_1^2 - x_2^2)^2).
$$
The critical values of $\hat V$ and $\hat K$ coincide with $\h_c=1/(4\sqrt{2})$ and the critical points of $\hat V$ are
\begin{equation}\label{equ: singularities}
\pm(2^{1/4}/\sqrt{3},0) \quad \mbox{ and } \quad \pm(0,2^{1/4}/\sqrt{3}).
\end{equation}
These are precisely the points of $\Pi_x(S).$

For each $0<h\leq \mathfrak h_c$, denote by $\mathcal H_h$ the energy surface projection $\Pi_x(\mathcal M_h)$, and denote $y+\hat F=r_h(x)(s+it)$, where
$$
r_h(x)^2=2(h-\hat V(x))= 2h - \frac{9}{16} (x_1^2 + x_2^2) (2 - 3 (x_1^2 - x_2^2)^2).
$$

We have
$$
\mathcal H_{h'} \subset \mathcal H_h \setminus \partial \mathcal H_h \qquad \forall 0 <h'<h\leq \mathfrak h_c.
$$
The expression for $G(\theta,x)$ in \eqref{G} is denoted $G_h(\theta,x)$ for each $h$, and the corresponding functions $c$ and $d$ are denoted $c_{h,t}(x)$ and $d_{h,s}(x)$, respectively, that is
\begin{equation}\label{G_2}
\begin{aligned}
G_h(\theta,x):= &r_h(x)^2c_{h,t}(x)d_{h,s}(x)+c_{h,t}(x)\hat V_2(x)^2+d_{h,s}(x)\hat V_1(x)^2\\ & -r_h(x)^2\hat V_{12}(x)^2-2\hat V_{12}(x)\hat V_1(x)\hat V_2(x),
\end{aligned}
\end{equation}
and
$$
\begin{aligned}
    c_{h,t}(x) &=r_h(x) t\hat f_{2,11}(x_1)+\hat V_{11}(x)=6tx_1 r_h(x)+\frac{9}{16}(2-45x_1^4+3x_2^4+18x_1^2x_2^2),\\
    d_{h,s}(x)& =r_h(x) s\hat f_{1,22}(x_2)+\hat V_{22}(x)=-6sx_2r_h(x)+\frac{9}{16}(2+3x_1^4-45x_2^4+18x_1^2x_2^2),
\end{aligned}
$$
where $s=\cos \theta$ and $t=\sin \theta$.

Our goal is to prove that for every $0< h\leq \mathfrak h_c$
\begin{equation}\label{Gh_positive}
G_h(\theta,x) >0, \qquad \forall x\in \mathcal H_h \setminus \Pi_x(S),\quad \theta \in \R / 2\pi \Z.
\end{equation}
Due to the symmetry of $\hat K$, it is enough to show that \eqref{Gh_positive} holds on $(\mathcal H_h \setminus \Pi_x(S)) \cap \{x_2\geq x_1 \geq 0\}$. Therefore, we assume in the computations below that $x\in \mathcal H_h$ lies in the second octant, that is, we only consider $x=(x_1,x_2)$ in the subset
$$
\mathcal H_h':=\{(x_1,x_2) \in \mathcal H_h \setminus \Pi_x(S),  x_2\geq x_1 \geq 0\},
$$
see Figure \ref{fig:Lunar_Hill_region}.

For every $0<h\leq \mathfrak h_c$, the boundary of $\mathcal H_h$ restricted to the first quadrant of the $x$-plane, can be written as the graph of a strictly decreasing function $x_2=x_2(x_1), x_1\in [0,R_h],$ satisfying $\hat V(x_1,x_2(x_1))=h$, where $0<R_h\leq R_{\mathfrak h_c}=2^{1/4}/\sqrt{3}$ is the unique solution of $\hat V(R_h,0)=\frac{9}{32}R_h^2(2-3R_h^4)=h$. This follows from
$$
\begin{aligned}
\hat V_1(x) = \frac{9}{16}x_1 (2 - 9 x_1^4 + 6 x_1^2 x_2^2 + 3x_2^4)>0,\\
\hat V_2(x)= \frac{9}{16}x_2 (2 + 3 x_1^4 + 6 x_1^2 x_2^2 - 9 x_2^4)>0,
\end{aligned}
$$
for every $(x_1,x_2)$ satisfying $0< x_1,x_2<2^{1/4}/\sqrt{3}$, where $\hat V_j = \partial_{x_j}V$. In particular, if $h=\mathfrak h_c$, we have $\hat V(2^{1/4}/3,2^{1/4}/3)=\mathfrak h_c$.

Since $r_h(x) \geq 0,$ $\hat f_{2,11}=6x_1\geq 0$ and $\hat f_{1,22}=-6x_2\leq 0$, we have
$$
\begin{aligned}
    c_{h,t}(x) & \geq c_{h,-1}(x)=-6r_h(x) x_1+\hat V_{11}(x), \qquad  \forall t\in[-1,1],\\
    d_{h,s}(x) &\geq d_{h,1}(x)=-6r_h(x) x_2+\hat V_{22}(x), \qquad \forall s\in [-1,1].
\end{aligned}
$$
Note that $\hat F$ and $\hat V$ are independent of $h$, and $\partial_h r_h,-\partial_hc_{h,-1},-\partial_hd_{h,1}\geq 0$. We identify $\partial \mathcal H_h$ with the corresponding curve of $\mathcal M_h$.

\begin{figure}[hbpt]
    \centering
    \includegraphics[width=0.4\linewidth]{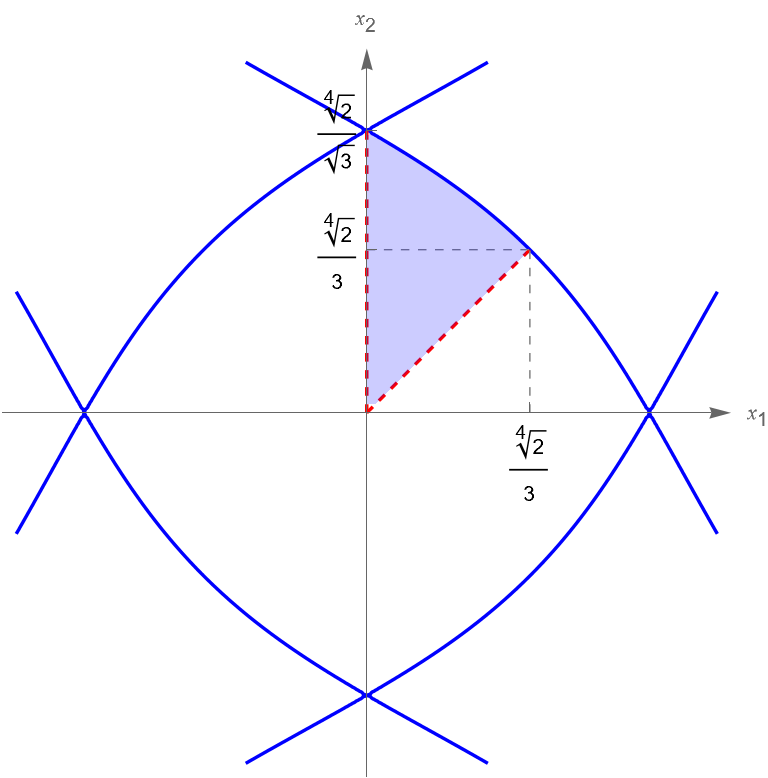}
    \caption{The Hill region $\mathcal H_{\mathfrak h_c}$ in the second octant.}
    \label{fig:Lunar_Hill_region}
\end{figure}

\begin{lem}\label{lem: ct positive}
$c_{h,t}(x)\geq c_{h,-1}(x)> 0$ holds for every $t\in[-1,1]$, $x\in \mathcal H_h'$ and $0<h\leq \h_c$.
\end{lem}
\begin{proof}
Recall that $\partial_hc_{h,-1}\leq 0$ for every $0<h\leq \mathfrak h_c$. Hence,
$c_{h,t}\geq c_{h,-1}\geq c_{\h_c,-1}$. It is thus sufficient to consider $h=\h_c$.
Denote $c_{-1}:=c_{\h_c,-1}$ and $r:=r_{\h_c}$.  Then
\begin{equation}\label{equ: c_-1 x2}
\partial_{x_2}c_{-1}(x)=\frac{27}{4} x_2 \left(3 x_1^2 + x_2^2 + \frac{x_1}{2r} (2 + 3 x_1^4 + 6 x_1^2 x_2^2 - 9 x_2^4)\right)\geq 0, \quad \forall x\in \mathcal H_{\h_c},  r>0.
\end{equation}
Therefore, $c_{-1}(x_1,x_2)\geq c_{-1}(x_1,x_1)=\frac{3}{8} (3 - 36 x_1^4 - 8 x_1(\sqrt 2 - 9 x_1^2)^{1/2})>0, \forall x\in \mathcal H_{\h_c}$, where the last inequality follows from $(3 - 36 (2^{1/4}/3)^4)^2 - (8 x_1 (\sqrt 2 - 9 x_1^2)^{1/2})^2 = \frac{361}{81} - 64 \sqrt 2 x_1^2 + 576 x_1^4\geq \frac{73}{81}> 0$ for every $0\leq x_1\leq 2^{1/4}/3$.
\end{proof}

\begin{lem}\label{lem: r^2d+v2^2 positive}
$r_h(x)^2d_{h,1}(x)+\hat V_2(x)^2\geq 0$ for every $x\in \mathcal H_h'$, $0<h\leq \h_c$.
\end{lem}
\begin{proof} Recall that $\partial_h d_{h,1}\leq 0$ and $\partial_hr_h\geq 0$. Hence
 $$
 \partial_h(d_{h,1}(x)+\hat V_2(x)^2/r_h(x)^2) =\partial_h d_{h,1}(x)-2\hat V_2(x)^2 \partial_h r_h/r_h(x)^3\leq 0,
 $$
 for every $x\in \mathcal H_h$ and $0<h\leq \h_c$. Hence, it is sufficient to consider $h=\h_c$.

Denote $c_{t}:=c_{\h_c,t}$,    $d_{s}:=d_{\h_c,s}$. We observe that $c_{-t}(x_2,x_1)=d_{t}(x_1,x_2)$. Then from \eqref{equ: c_-1 x2}, we see that $d_1(x)$ increases in $x_1$ and thus $d_1(x_1,x_2)\geq d_1(0,x_2)$ for $x_1\geq 0$. Hence, it is enough to prove that
$$
W_1(x):=r(x)^2d_1(0,x_2)+\hat V_2(x)^2\geq 0
$$
for every $x\in \mathcal H_{\h_c}'$. We compute
$$\begin{aligned}
&\quad\ 32\sqrt 2 W_1(2^{1/4} \sqrt{l_1/3}, 2^{1/4} \sqrt{l_2/3})\\
&=3 l_2 (3 + l_1^2 + 2 l_1 l_2 - 3 l_2^2)^2 + \left(9 (1 - 5 l_2^2) -  8 (1 - l_2) \sqrt{3 l_2 (2 + l_2)}\right) (2 - (l_1 + l_2) (3 - (l_1 - l_2)^2))\\
&\geq 3 l_2 (3 + l_1^2 + 2 l_1 l_2 - 3 l_2^2)^2 + (9 (1 - 5 l_2^2) -  8 (1 - l_2) (1 + 2 l_2) ) (2 - (l_1 + l_2) (3 - (l_1 - l_2)^2))\\
&=:W_2(l_1,l_2),
\end{aligned}$$
for every $0\leq l_1,l_2\leq 1$ so that $(2^{1/4} \sqrt{l_1/3}, 2^{1/4} \sqrt{l_2/3})\in \mathcal H'_{\mathfrak h_c}$. Here, we have used that $(1 + 2 l_2)^2 - 3 l_2(2 + l_2) = (1 - l_2)^2\geq 0$ and that
$$
8\sqrt 2(\mathfrak h_c-\hat V(2^{1/4} \sqrt{l_1/3}, 2^{1/4} \sqrt{l_2/3}))=2 - (l_1 + l_2) (3 - (l_1 - l_2)^2)\geq 0.$$
Finally, we compute
$$\begin{aligned}
W_2(l_1,l_2)&=l_1^2 l_2 (W_{21}+ 3 l_1^2) + 2 (1 + 7 l_2 + l_2^2) (1 -l_1)^3 (1 - l_2)^3 \\
&\quad\ + (3 + 48 l_2 + 20 l_2^2 + 110 l_2^3 - 31 l_2^4 - 6 l_2^5) l_1 (1 - l_1)^2,\quad \forall 0\leq l_1,l_2\leq 1.
\end{aligned}$$
where
$W_{21}(l_1, l_2) = 89 + 150 l_2 + 141 l_2^2 - 38 l_2^3 - 6 l_2^4 + l_1 (-48 - 71 l_2 - 76 l_2^2 + 23 l_2^3 + 4 l_2^4)$ satisfies $W_{21}(0,l_2)>0$ and $W_{21}(1,l_2)=41 + 79 l_2 + 65 l_2^2 - 15 l_2^3 - 2 l_2^4>0$ for every $0\leq l_2\leq 1$. Also, we observe that $3 + 48 l_2 + 20 l_2^2 + 110 l_2^3 - 31 l_2^4 - 6 l_2^5>0,\, \forall 0\leq l_2\leq 1$. Therefore, $W_2(l_1,l_2)\geq 0$ for $0\leq l_1,l_2\leq 1$ and this implies that $W_1(x)\geq 0$ for every $x\in \mathcal H_{\h_c}'$.
\end{proof}

By Lemmas \ref{lem: ct positive} and \ref{lem: r^2d+v2^2 positive}, we obtain
$$\begin{aligned}
G_h(\theta, x)&\geq c_{h,t}(x)(r_h(x)^2d_{h,1}(x)+\hat V_{2}(x)^2)+d_{h,1}(x)\hat V_{1}(x)^2-r_h(x)^2\hat V_{12}(x)^2-2\hat V_{12}(x)\hat V_{1}(x)\hat V_{2}(x)\\
&\geq c_{h,-1}(x)(r_h(x)^2d_{h,1}(x)+\hat V_{2}(x)^2)+d_{h,1}(x)\hat V_{1}(x)^2-r_h(x)^2\hat V_{12}(x)^2-2\hat V_{12}(x)\hat V_{1}(x)\hat V_{2}(x)\\
&=:\hat E_{0,h}(x),\quad \forall x\in \mathcal H_{h}',\ 0<h\leq \h_c.
\end{aligned}$$

We aim to prove that $\hat E_{0,h}(x)>0$ for every $x\in \mathcal H_{h}'$ and $0<h\leq \h_c$. Define
$$
\begin{aligned}
W(x,r) &
:=\hat V_{2}(x)^2+\frac{\tilde d_{1}(x,r)\hat V_{1}(x)^2-2\hat V_{12}(x)\hat V_{1}(x)\hat V_{2}(x)}{\tilde c_{-1}(x,r)},\\
W_{h}(x) &:=W(x,r_h(x)),
\end{aligned}
$$
for every $x\in \mathcal H_h, 0<h\leq \h_c$, where
$$
\begin{aligned}
\tilde c_{t}(x,r) &:=r t\hat f_{2,11}(x_1)+\hat V_{11}(x)=6x_1\cdot rt+\frac{9}{16}(2-45x_1^4+3x_2^4+18x_1^2x_2^2),\\
\tilde d_{s}(x,r)
& :=r s\hat f_{1,22}(x_2)+\hat V_{22}(x)=-6x_2\cdot rs+\frac{9}{16}(2+3x_1^4-45x_2^4+18x_1^2x_2^2).
\end{aligned}
$$
In particular, $c_{h,t}(x)=\tilde c_t(x,r_h(x))$ and $d_{h,s}(x)=\tilde d_s(x,r_h(x))$.


Denote
$$
\hat E_{1,h}(x):=d_{h,1}(x)-\frac{\hat V_{12}(x)^2}{c_{h,-1}(x)}+\frac{W_h(x)}{r_h(x)^2},\quad \forall x\in \mathcal H_{h}', r_h(x)>0.
$$
We see that $\hat E_{0,h}(x)=r_h^2(x)c_{h,-1}(x)\hat E_{1,h}(x)$. Moreover, we compute
$$
\begin{aligned}
\partial_h \hat E_{1,h}(x) &=\partial_h d_{h,1}(x)+\frac{\hat V_{12}(x)^2}{c_{h,-1}^2(x)}\partial_h c_{h,-1}(x)\\ &-\frac{2\partial_h r_h(x)}{r_h^3(x)}W_h(x)+\frac{1}{r_h(x)^2}\partial_h W_h(x).
\end{aligned}
$$
We need the following lemma.
\begin{lem}\label{lem: prop of Wh}
The following statements hold:
\begin{itemize}
\item[(i)] $W_h(x),-\partial_{h} W_h(x)\geq 0$ for every $x\in \mathcal H_{h}'\setminus \partial\mathcal H_{h}$, $0<h\leq \h_c$,
\item[(ii)] $W_h(x)>0$ for every $x\in \mathcal H_{h}'\cap \partial \mathcal H_{h}$, $0<h\leq \h_c$.
\end{itemize}
\end{lem}

For now, we skip the proof of Lemma \ref{lem: prop of Wh}. Recall that $\partial_h r_h(x),-\partial_h d_{h,1}(x),-\partial_h c_{h,-1}(x)\geq 0$ for every $x\in \mathcal H_{h}'$, $0<h\leq \h_c$.
By Lemma~\ref{lem: prop of Wh}-(i), we have
$W_h(x)\geq W_{\h_c}(x)$ for every $x\in \mathcal H_{h}'\subset \mathcal H_{\h_c}'$. Hence $\hat E_{1,h}(x)\geq \hat E_{1,\h_c}(x)$ for every
$x\in \mathcal H_{h}'\setminus \partial \mathcal H_{h}\subset \mathcal H_{\h_c}'\setminus \partial \mathcal H_{\h_c}$. Using Lemma \ref{lem: ct positive}, we conclude that $\hat E_{0,h}(x)>0, \forall x\in \mathcal H_{h}'\setminus\partial\mathcal H_{h}$ if $\hat E_{0,\h_c}(x)>0, \forall x\in \mathcal H_{\h_c}'\setminus\partial \mathcal  H_{\h_c}$. Moreover, for every $x\in \mathcal H_{h}'\cap \partial\mathcal H_{h}$, we have
$$
\hat E_{0,h}(x)=c_{h,-1}(x,0)\hat V_2(x)^2+d_{h,1}(x,0)\hat V_1(x)^2-2\hat V_{12}(x)^2\hat V_1(x)\hat V_2(x)=c_{h,-1}(x,0)W_h(x).
$$
By Lemmas \ref{lem: ct positive} and \ref{lem: prop of Wh}-(ii), we obtain $\hat E_{0,h}(x)>0$ for every $x\in\mathcal H_{h}'\cap \partial\mathcal H_{h}$.
It remains to prove the following lemma.

\begin{lem}\label{lem: hat E 0,hc is positive}
 $\hat E_{0,\h_c}(x)>0$ for every $x\in \mathcal H_{\h_c}'$.
\end{lem}

By Lemma \ref{lem: hat E 0,hc is positive}, we conclude that $\hat E_{0,h}(x)>0$ for every $x\in \mathcal H_{h}'$,  $0<h\leq \h_c$. Up to proving Lemmas \ref{lem: prop of Wh} and \ref{lem: hat E 0,hc is positive}, the proof of Proposition \ref{prop_strict_convexity} is complete.

\begin{proof}[Proof of Lemma \ref{lem: prop of Wh}]
We first prove that $-\partial_h W_h(x)\geq 0$ for every $x\in \mathcal H_{h}'\setminus \partial \mathcal H_{h}$, $0<h\leq \h_c$. Recall that $W_h(x)=W(x,r_h(x))$. We compute
$$
- \partial_r W(x,r_h(x))=\frac{243}{128}\cdot \frac{x_1\hat V_1(x)}{c_{h,-1}^2(x)}W_{3}(x),
$$
where
$$\begin{aligned}
W_{3}(x):&=27 x_1^9 + 405 x_1^8 x_2 + 216 x_1^7 x_2^2 - 432 x_1^6 x_2^3 + 144 x_1^3 x_2^6 - 54 x_1^4 x_2 (2 + x_2^4)\\
& + 24 x_1^2 x_2^3 (2 + 3 x_2^4) - 6 x_1^5 (-2 + 51 x_2^4) + x_1 (-4 + 132 x_2^4 - 81 x_2^8) + x_2 (2 + 3 x_2^4)^2.
\end{aligned}$$

Let us define
$$
\begin{aligned}
W_{31}(c,s_2):= & 9 (1 - c) + 6 (1 + 11 c + 4 c^2 - 9 c^4 + c^5) s_2^4 \\ &+ (1 + c)^2 (1 - c) (1 - 10 c + 19 c^2 - 12 c^3 + 15 c^4 - 42 c^5 - 3 c^6) s_2^8.
\end{aligned}
$$
Observe that
$$
s_2 W_{31}(c,s_2)= \frac{9 \sqrt 3}{4\cdot 2^{1/4}} W_{3}\left(\frac{2^{1/4}}{\sqrt 3} cs_2, \frac{2^{1/4}}{\sqrt 3} s_2\right),\quad 0\leq c,s_2\leq 1.
$$

If $1 - 10 c + 19 c^2 - 12 c^3 + 15 c^4 - 42 c^5 - 3 c^6\geq 0$, for some $0\leq c\leq 1$, then $W_{31}(c,s_2)\geq 0$. If $1 - 10 c + 19 c^2 - 12 c^3 + 15 c^4 - 42 c^5 - 3 c^6< 0$, for some $0< c\leq 1,$ then $W_{31}(c,s_2)$ is a concave function of $s_2^4$. Since $W_{31}(c,0)=9(1-c)\geq 0$ and $W_{31}(c,1)=16 + 4 (12 + 20 c + 24 c^2 + 9 c^3 + 2 c^4 - 10 c^5 - 4 c^6) (1 - c) c + 29 c^8 + 3 c^9>0$ for every $0\leq c\leq 1$, we conclude that $W_{31}(c,s_2)\geq 0$ for every $0\leq c,s_2\leq 1$. This implies that $-\partial_r W(x,r_h(x))\geq 0$ for every $x\in \mathcal H_{h}'$, $0<h\leq \h_c$. This follows from the fact that
$$
\left\{\hat V\left(\frac{2^{1/4}}{\sqrt 3} cs_2, \frac{2^{1/4}}{\sqrt 3} s_2\right)\leq \h_c, s_2\geq cs_2\geq 0\right\}\subset \{0\leq c,s_2\leq 1\}.
$$
In particular, $-\partial_h W_h(x)\geq 0$ for every $x\in \mathcal H_{h}'\setminus \partial \mathcal H_{h}$, $0<h\leq \h_c$.

Now we prove that $W_h(x)\geq 0$ for every $x\in \mathcal H_{h}'\setminus \partial \mathcal H_{h},\, 0<h\leq \h_c$, and that $W_h(x)>0$ for every $x\in  \mathcal H_{h}'\cap \partial \mathcal H_h$, $0<h\leq \h_c$.

From the estimate $\partial_h W_h \leq 0$ proved above, we have $W_h(x)=W(x,r_h(x))\geq W(x,r_{\h_c}(x))=W_{\h_c}(x)$ for every $x\in \mathcal H_{h}'\subset \mathcal H_{\h_c}'$, $0<h\leq \h_c$. It is thus sufficient to prove that $W_{\h_c}(x)>0$ for every $x\in \mathcal H_{\h_c}'\setminus \{0\}$. One readily checks that $W_{\h_c}(0,2^{1/4}/\sqrt{3})=0$ and $W_h(0)=0$ for every $h$. Let $\hat W_{\h_c}(u,v):=W_{\h_c}(u-v,u+v)$. Then
$$
\{(u-v,u+v)\in \mathcal H_{\h_c}'\}\subset \left\{0\leq u\leq \frac{2^{1/4}}{3},0\leq v\leq \frac{2^{1/4}}{2\sqrt 3}\right\}.
$$
Choose new coordinates $(c,l_2)\in [-1/3,1]\times [0,1]$ so that
$$
(u,v)=\left(\frac{2^{1/4}}{2\sqrt 3}\sqrt{l_1},\frac{2^{1/4}}{2\sqrt 3}\sqrt{l_2}\right)=\left(\frac{2^{1/4}}{2\sqrt 3}\sqrt{1+c(l_2-1)},\frac{2^{1/4}}{2\sqrt 3}\sqrt{l_2}\right).
$$
We compute $\h_c-\hat V(u-v,u+v)= (1-l_2)^2 (1 + 3 c + c l_2 + c^2 l_2)$. Hence, $\mathcal H_{\h_c}'$ is contained in $1+3c+cl_2+c^2l_2 \geq 0$ and the boundary $\mathcal H_{\h_c}'\cap \partial \mathcal H_{\h_c}$ is given by
$$
\begin{aligned}
&\left\{(c,l_2)| l_2=\bar l_2(c):=-\frac{1 + 3 c}{c (1 + c)},\ -1/3\leq c\leq -2+\sqrt{3}\right\}\quad \text{or}\\
&\left\{(c,l_2)| c=\underline c(l_2):=\frac{1}{2 l_2}\left(-3 - l_2 + \sqrt{9 + 2 l_2 + l_2^2}\right),\ 0\leq l_2\leq 1\right\},
\end{aligned}
$$
see Figure \ref{fig: Lunar_Hill_region_cors}.
\begin{figure}[hbpt]
    \centering
    \begin{minipage}[t]{0.3\textwidth}
    \centering
    \includegraphics[width=\textwidth]{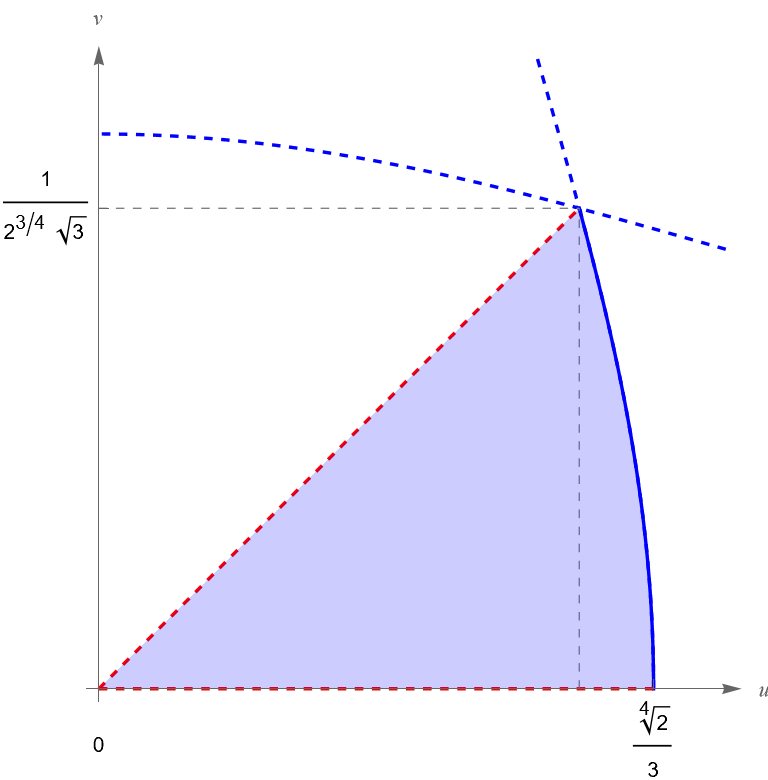}
    \end{minipage}
    \quad
    \begin{minipage}[t]{0.3\textwidth}
    \centering
    \includegraphics[width=\textwidth]{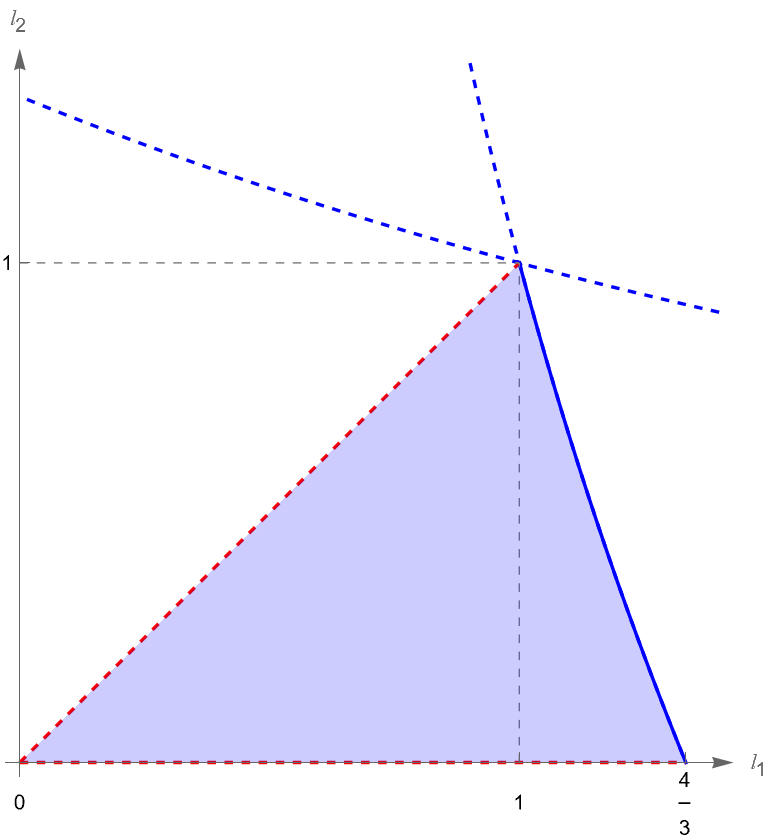}
    \end{minipage}
    \quad
    \begin{minipage}[t]{0.33\textwidth}
    \centering
    \includegraphics[width=\textwidth]{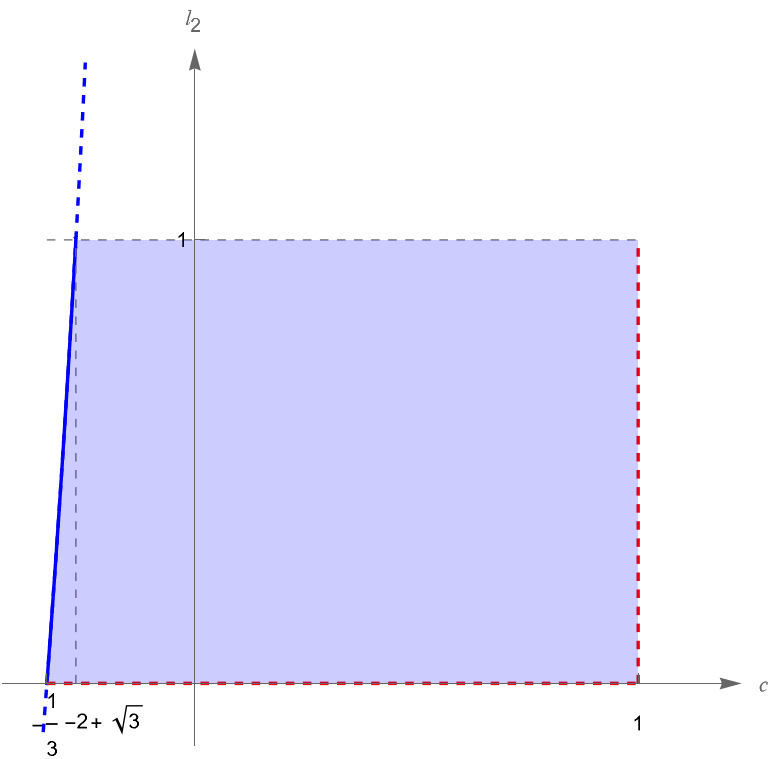}
    \end{minipage}
    \caption{The region $\mathcal H_{\h_c}'$ in different coordinates}
    \label{fig: Lunar_Hill_region_cors}
\end{figure}

In coordinates $(c,l_2)$, we have
\begin{equation}\label{equ: hat W_hc}
\begin{aligned}
\frac{512\sqrt{2}}{3}\hat W_{\h_c}\left(\frac{2^{1/4}}{2\sqrt 3}\sqrt{1+c(l_2-1)},\frac{2^{1/4}}{2\sqrt 3}\sqrt{l_2}\right)=(1-l_2)^2(W_a(c,l_2) (1-l_2)^3+l_2\hat W_{\h_c,1}(c,l_2)),
\end{aligned}
\end{equation}
where
$$
\begin{aligned}
\hat W_{h_c,1}(c,l_2)&:=W_b(c,l_2)+2 (1 - c) (1 - l_2) W_c(c,l_2)\cdot \sqrt{6 (1 + c (l_2-1)) (1 + 3 c + c l_2 + c^2 l_2)},\\
W_a(c,l_2)&:=9 (1 - c) \Big[6 + 6 c - 3 c^2 (1 - l_2) - 2 \sqrt{6 (1 + c (l_2-1)) (1 + 3 c + c l_2 + c^2 l_2)}\Big],\\
W_b(c,l_2)&:=144 (1 + 4 c + c^2) + 9 (1 - l_2) \big[8 - 28 c - 76 c^2 + 21 c^3 - 3 c^4 - (2 c^2 + c^3 + c^4 + 2 c^5) l_2^3\\
&\quad\ - (8 + 8 c + 2 c^2 + 35 c^3 + 13 c^4 + 2 c^5) l_2 - (4 c + 8 c^2 + c^3 + 9 c^4 - 4 c^5) l_2^2\big],\\
W_c(c,l_2)&: =-4 (1 - c) - (1 - l_2) (12 + (13 + 3 l_2) c + (3 + 5 l_2 - (1 - l_2) c) c^2).
\end{aligned}$$

Notice that $12 r_h(x) (x_1 + x_2)=(1 - l_2)\sqrt{6 (1 + c (l_2-1)) (1 + 3 c + c l_2 + c^2 l_2)}$.
It is straightforward to check that $W_c(c,l_2)\leq 0$ for every $(c,l_2)\in[-1/3,1]\times [0,1]$.

For the function $W_a$, we take the square of the corresponding non-negative terms and obtain
$$\begin{aligned}
&\quad\ (6 + 6 c - 3 c^2 (1 - l_2))^2 - 24 (1 + c (l_2-1)) (1 + 3 c + c l_2 + c^2 l_2)\\
&=3 (4 + 8 c + 24 c^2 - 12 c^3 + 3 c^4) -
 6 c (8 + 6 c - 10 c^2 + 3 c^3) l_2 - 3 c^2 (8 + 8 c - 3 c^2) l_2^2=:f_a(c,l_2).
\end{aligned}$$
Notice that $f_a(c,l_2)$ is concave in $l_2$ for every $-1/3\leq c\leq 1$.  Since $f_a(c,0)>0$ and $f_a(c,1)=12(1-c)^2\geq 0$ for every $-1/3\leq c\leq 1$, we conclude that $W_a(c,l_2)\geq 0$ for every $(c,l_2)\in[-1/3,1]\times [0,1]$.

Then, to prove that $\hat W_{\h_c}(x)>0$ for every $x\in \mathcal H_{\h_c}'\setminus \{0\}$, it suffices to prove that $\hat W_{\h_c,1}(c,l_2)>0$ for every $\underline c(l_2)\leq c\leq 1$ and $0\leq l_2<1$. Notice that the point $(0,2^{1/4}/\sqrt{3})\in \Pi_x(S)$ blows up to the line segment $[-2+\sqrt{3},1]\times \{l_2=1\}$. We split the proof into two steps:

\textbf{Step 1}: We claim that $\hat W_{\h_c,1}(c,l_2)>0$ for every $\underline c(l_2)\leq c\leq 0$ and $0\leq l_2< 1$. Using that $W_c \leq 0$ and that $\sqrt{6 (1 + c (l_2-1)) (1 + 3 c + c l_2 + c^2 l_2)}\leq (5/2) (1 + (3/4 + l_2) c - 3 c^2)$, we obtain
$$
\hat W_{\h_c,2}(c,l_2):=W_b(c,l_2) +
 2 (1 - c) (1 - l_2) W_c(c,l_2)\cdot (5/2) (1 + (3/4 + l_2) c - 3 c^2)\leq \hat W_{\h_c,1}(c,l_2).
$$
We have used that $\sqrt{6}<5/2$ and that
\begin{equation}\label{equ: enlarge the sqare root 1}
\begin{aligned}
&\quad\ 6 (1 + (3/4 + l_2) c - 3 c^2)^2-6 (1 + c (l_2-1)) (1 + 3 c + c l_2 + c^2 l_2)\\
&=(-c) \big[3 + (9/8) c (13 + 8 l_2) + 3 c^2 (9 + 10 l_2 + 2 l_2^2) - 54 c^3\big]\geq 0,\quad \forall -1/3\leq c\leq 0.
\end{aligned}
\end{equation}
Indeed, notice that $3 + (9/8) c (13 + 8 l_2) + 3 c^2 (9 + 10 l_2 + 2 l_2^2)$ attains the minimum at $c=-\frac{3 (13 + 8 l_2)}{16 (9 + 10 l_2 + 2 l_2^2)}$ with minimal value $\frac{3 (783 + 688 l_2 - 64 l_2^2)}{256 (9 + 10 l_2 + 2 l_2^2)}>0$ for every $0\leq l_2\leq 1$.

Now we check the positivity of $\hat W_{\h_c,2}$. We compute its partial derivative with respect to $c$
$$
\begin{aligned}
108\partial_c\hat W_{\h_c,2}(c,l_2)&=
  (32292 + 17550 c_1 + 4554 c_1^2 + 1712 c_1^3 - 425 c_1^4 - 40 c_1^5) (1 - l_2)^4\\
&+ (1 - l_2) l_2 \big[4 (65610 + 14778 c_1 - 3123 c_1^2 - 808 c_1^3 - 200 c_1^4) (1 - l_2)l_2\\
&+ \frac{1}{3} (439344 + 176634 c_1 + 8586 c_1^2 - 288 c_1^3 - 4135 c_1^4 - 120 c_1^5) (1 -l_2)^2\\
&+ 12 (17541 + 636 c_1 - 897 c_1^2 - 80 c_1^3) l_2^2\big] + 10368 (6 - c_1) l_2^4\\ &>0,\quad \forall\ c_1:=-3c,l_2\in[0,1].
\end{aligned}
$$
We further compute along the boundary $\mathcal H_{\h_c}'\cap \partial \mathcal H_{\h_c}$ that
$$
\begin{aligned}
&\quad\ W^{b}_{\h_c,2}(c):=\frac{4 c^2 (1 + c)^4}{1 + 4 c + c^2}\hat W_{\h_c,2}(c,\bar l_2(c))\\
&=216 + 1629 c + 3507 c^2 + 1608 c^3 - 681 c^4 - 1554 c^5 - 1799 c^6 + 208 c^7 + 553 c^8 + 285 c^9 + 60 c^{10}.
\end{aligned}
$$
Since $(W^{b}_{\h_c,2})''(c)=(2/729) (2556603 - 1172232 c_1 - 330966 c_1^2 + 419580 c_1^3 - 242865 c_1^4 - 13104 c_1^5 + 15484 c_1^6 - 3420 c_1^7 + 300 c_1^8)>0$ for every $c_1=-3c\in [0,1]$,  $W^{b}_{\h_c,2}(-1/3)=-27424/19683<0$ and $W^{b}_{\h_c,2}(-2+\sqrt{3})=48 (64867 - 37451 \sqrt 3)<0$, we conclude that $W^{b}_{\h_c,2}(c)<0$ for every $c\in [-1/3,-2+\sqrt{3}]$. Therefore, $\hat W_{\h_c,2}(c,\bar l_2(c))>0$ for every $c\in[-1/3,-2+\sqrt{3})$, since $1+4c+c^2$ is negative for every $c\in[-1/3,-2+\sqrt{3})$ and vanishes at $c=-2+\sqrt{3}$. Finally, we conclude that $\hat W_{\h_c,2}(c,l_2)\geq \hat W_{\h_c,2}(\underline c(l_2),l_2)= \hat W_{\h_c,2}(c,\bar l_2(c))>0$ for every $\underline c(l_2)\leq c\leq 0$ and $0\leq l_2<1$. In particular, we have $\bar l_2(-2+\sqrt{3})=1$ and $\hat W_{\h_c,2}(-2+\sqrt{3},1)=0$. Hence, the first claim holds.

\textbf{Step 2}: We claim that $\hat W_{\h_c,1}(c,l_2)>0$ for every $(c,l_2)\in [0,1]\times [0,1]$. As in Step $1$, we define
$$
\hat W_{\h_c,3}(c,l_2):=W_b(c,l_2) + 2 (1 - c) (1 - l_2) W_c(c,l_2)\cdot (5/2) (1 + (1 + l_2) c + (-4/3 + l_2) c^2)\leq \hat W_{\h_c,1}(c,l_2).
$$
The last inequality follows from $W_c(c,l_2)\leq 0$ on $[-1/3,1]\times [0,1]$, $\sqrt{6}<5/2$ and
\begin{equation}\label{equ: enlarge the square root 2}
\begin{aligned}
&\quad\ 6 (1 + (1 + l_2) c + (-4/3 + l_2) c^2)^2 - 6 (1 + c(l_2-1)) (1 + 3 c + c l_2 + c^2 l_2)\\
&= c^2 \big[(2/3)(12 + 9 l_2)(1 - c)^2 + (4/3) (2 - 6 l_2 + 9 l_2^2) c^2 + 2 l_2 (7 + 3 l_2) c (1 - c)\big]\\
&\geq 0,\ \ \forall 0\leq c,l_2\leq 1.
\end{aligned}
\end{equation}
We write $\hat W_{\h_c,3}$ as
$$\begin{aligned}
3\hat W_{\h_c,3}(c,l_2)&= (408 + 837 c - 1105 c^2 + 577 c^3 - 156 c^4 - 95 c^5 + 20 c^6) (1 - l_2)^4\\
& + (1 - l_2) l_2 \big[8 (306 + 900 c - 411 c^2 - 47 c^3 + 67 c^4 - 5 c^5) (1 - l_2) l_2 \\
&+ (1620 + 3933 c - 3522 c^2 + 841 c^3 + 303 c^4 - 264 c^5 + 5 c^6) (1 - l_2)^2 \\
&+ 4 (417 + 1458 c - 112 c^2 - 148 c^3 + 5 c^4) l_2^2 \big] + 432 (1 + 4 c + c^2) l_2^4\\
&>0,\quad \forall 0\leq c,l_2\leq 1.
\end{aligned}$$
We conclude that $\hat W_{\h_c,1}(c,l_2)\geq \hat W_{\h_c,3}(c,l_2)>0$ for every $(c,l_2)\in [0,1]\times [0,1]$, proving the claim.

Finally, using \eqref{equ: hat W_hc}, we conclude that $W_{\h_c}(x)>0$ on $\mathcal H_{\h_c}'\setminus\{0\}$. Lemma \ref{lem: prop of Wh} follows.
\end{proof}

\begin{proof}[Proof of Lemma \ref{lem: hat E 0,hc is positive}]
Let $\hat E_{c}(u,v):=\hat E_{0,\h_c}(u-v,u+v)$, where $\{(u-v,u+v)\in \mathcal H_{\h_c}'\}\subset\{0\leq u\leq 2^{1/4}/3,0\leq v\leq 2^{1/4}/(2\sqrt{3})\}$. In coordinates $(c,l_2)\in[-1/3,1]\times[0,1]$ determined by $(u,v)=(\frac{2^{1/4}}{2\sqrt 3}\sqrt{1+c(l_2-1)},\frac{2^{1/4}}{2\sqrt 3}\sqrt{l_2})$ as before, we have
$$
\frac{256\sqrt{2}}{3}\hat E_{c}\left(\frac{2^{1/4}}{2\sqrt 3}\sqrt{1+c(l_2-1)},\frac{2^{1/4}}{2\sqrt 3}\sqrt{l_2}\right)=(1 - l_2)^3\hat E_{c,1}(c,l_2),
$$
where
$$
\begin{aligned}
\hat E_{c,1}(c,l_2) &:=E_a(c,l_2) + E_b(c,l_2)\cdot \sqrt{6 (1 + c(l_2-1)) (1 + 3 c + c l_2 + c^2 l_2)}\\
E_a(c,l_2)& :=40 - 14 l_2 + 10 l_2^2 + c (56 + 58 l_2 + 10 l_2^2 + 20 l_2^3) + c^2 (-6 + 99 l_2 + 10 l_2^2 + 31 l_2^3 + 10 l_2^4) \\
& + c^3 (-36 + 16 l_2 + 53 l_2^2 - 4 l_2^3 + 7 l_2^4) + c^4 (3 l_2 - 13 l_2^2 + 11 l_2^3 - l_2^4) + c^5 (2 l_2^2 - 4 l_2^3 + 2 l_2^4),\\
E_b(c,l_2)& :=-l_2 (1 - l_2) \left(22 + 22 c + 19 c^2 - 3 c^3 + (-4 c - c^2 + c^3) l_2\right) - 12 (1 + c)^2 l_2^2\\
& -(14 + 8 c + 5 c^2 - 3 c^3) (1 - l_2)^2.
\end{aligned}$$

We immediately see that $E_b(c,l_2)\leq 0$ for every $(c,l_2)\in [-1/3,1]\times [0,1]$, since $22 + 22 c + 19 c^2 - 3 c^3 + (-4 c - c^2 + c^3) l_2$ and $14 + 8 c + 5 c^2 - 3 c^3$ are positive.

To show that $\hat E_{c,1}(c,l_2)>0$ for every $(c,l_2)\in[-1/3,1]\times [0,1]$, we split into two steps:

{\bf Step 1.} We claim that $\hat E_{c,1}(c,l_2)> 0$ for every $(c,l_2)\in [-1/3,0]\times [0,1]$. As before, see \eqref{equ: enlarge the sqare root 1}, we have
$$
\begin{aligned}
\hat E_{c,1a}(c,l_2):&=E_a(c,l_2)+E_b(c,l_2)\cdot (5/2)(1 + (3/4 + l_2) c - 3 c^2)\\
&\leq \hat E_{c,1}(c,l_2),\quad \forall (c,l_2)\in[-1/3,0]\times [0,1].
\end{aligned}
$$

We rewrite $\hat E_{c,1a}$ as
$$\begin{aligned}
\hat E_{c,1a}(c,l_2)&=\frac{1}{36} \big[9 (1015 - 476 l_2 - 5 l_2^2 + 24 l_2^3) + c (-90 + 90 l_2 + 38 l_2^2 - 46 l_2^3 + 8 l_2^4)\big] c^2 (1 + 3 c)^2 \\
&+ \frac{1}{72} (49041 - 31095 l_2 - 1265 l_2^2 + 43 l_2^3 - 308 l_2^4) (-c^3) (1 + 3 c) + (5 + l_2) (1 + 3 c)^4\\
&+ \frac{1}{6} (4617 - 2664 l_2 - 961 l_2^2 - 160 l_2^3 - 76 l_2^4) c^4 + \frac{1}{4} (201 - 29 l_2 - 10 l_2^2) (-c) (1 + 3 c)^3,
\end{aligned}$$
and we see that $\hat E_{c,1a}(c,l_2)>0$ for every $(c,l_2)\in [-1/3,0]\times [0,1]$.

{\bf Step 2.} Now we claim that $\hat E_{c,1}(c,l_2)> 0$ for every $(c,l_2)\in[0,1]\times [0,1]$. As before, see \eqref{equ: enlarge the square root 2}, we have
$$
\begin{aligned}
\hat E_{c,1b}(c,l_2):&=E_a(c,l_2)+E_b(c,l_2)\cdot (5/2) (1 + (1 + l_2) c + (-4/3 + l_2) c^2)\\
&\leq \hat E_{c,1}(c,l_2),\quad \forall (c,l_2)\in [0,1]\times[0,1].
\end{aligned}
$$
Again, we rewrite $\hat E_{c,1b}$ as
$$
\begin{aligned}
\hat E_{c,1b}(c,l_2):&=\frac{1}{3} \big[2 (7 + 11 l_2 - 3 l_2^2) (3 + l_2^2) + (105 + 6 l_2 + 223 l_2^2 - 4 l_2^3 - 30 l_2^4) (1 - c)\big] c^4\\
&+ (5 + l_2 + 3 (7 + 9 l_2) c) (1 - c)^4 + \frac{1}{6} (373 + 531 l_2 + 95 l_2^2 - 39 l_2^3)(1 - c)^2 c^2\\
&+ \frac{1}{6} \big[3 (1 - l_2) (8 - 7 l_2) + 412l_2^2 - 67 l_2^3 - 33 l_2^4\big] (1 - c)^2 c^3,
\end{aligned}
$$
and we see that $\hat E_{c,1b}(c,l_2)>0$ for every $(c,l_2)\in [0,1]\times [0,1]$, proving the claim.

Combining the estimates above for $\hat E_{c,1a}$ and $\hat E_{c,1b}$, we conclude that
$$
\hat E_{c,1}(c,l_2)>0,\quad \forall \underline c(l_2)\leq c\leq 1,\ 0\leq l_2\leq 1.
$$
Recall that $\mathcal H_{\h_c}'\subset \{(c,l_2)|\, \underline c(l_2)\leq c\leq 1, 0\leq l_2\leq 1\}$.
Then $\hat E_{0,\h_c}(0,2^{1/4}/\sqrt 3)=0$ and $\hat E_{0,\h_c}(x)>0$ for every $x\in \mathcal H_{\h_c}'$. Lemma \ref{lem: hat E 0,hc is positive} follows.
\end{proof}

The proof of Proposition \ref{prop_strict_convexity} is now complete.

\section{Proof of Proposition \ref{prop_contactform_J}}
\label{sec_contactform}
We work with the regularized Hamiltonian
\[
\widehat K(y,x)=\frac{(y_1-x_2^3)^2}{2}
+\frac{(y_2+x_1^3)^2}{2}+\frac{9}{32}(x_1^2+x_2^2)\bigl(2-3(x_1^2-x_2^2)^2\bigr),
\]
where \(\omega_0=\sum_{j=1}^2dy_j\wedge dx_j\). Its critical value is
$\h_c=1/(4\sqrt2).$

By Proposition~\ref{prop_strict_convexity}, the bounded component of
$\widehat K^{-1}(h)$ is strictly convex for every $0< h\leq \h_c$. Consequently, the standard radial Liouville vector field
$$
Y_0:=\frac12\sum_{j=1}^2
\left(y_j\partial_{y_j}+x_j\partial_{x_j}
\right)
$$
is transverse to the regular part of the critical component. By compactness, after fixing sufficiently small neighborhoods of the critical points, $Y_0$ remains transverse to all nearby levels
$
\widehat K^{-1}(h),
h\in[\h_c,\h_c+\varepsilon_0],
$ $\varepsilon_0>0$ small,
outside these neighborhoods.

The critical level contains four saddle-center points corresponding to the critical points of the potential,
$(\pm\ 2^{1/4}/\sqrt3,0),
(0,\pm 2^{1/4}/\sqrt3).$
We first analyze the point near
$(-2^{1/4}/\sqrt{3},0).$

Let $a=2^{1/4}/\sqrt3$. The corresponding critical point of $\widehat K$ is $(0,a^3, -a,0).$

We work in the translated coordinates
\[
\hat y_1=y_1,
\qquad
\hat y_2=y_2-a^3,
\qquad
\hat x_1=x_1+a,
\qquad
\hat x_2=x_2.
\]
Thus, the critical point corresponds to the origin in the coordinates
$
(\hat y_1,\hat y_2,\hat x_1,\hat x_2).
$

The quadratic part of \(\widehat K-\h_c\) is
\[
Q(\hat y,\hat x)
=
\frac12 \hat y_1^2+\frac12 \hat y_2^2
+\sqrt2\,\hat y_2\hat x_1
-\frac54 \hat x_1^2
+\frac34 \hat x_2^2.
\]

The Hamiltonian matrix $J\nabla^2Q(0,a^3,-a,0)$ associated with \(Q\) has eigenvalues
$\pm\lambda_1$ and
$\pm i\lambda_2,$
where
\[
\lambda_1=\sqrt{\sqrt7+\frac12}\approx 1.7736,
\qquad
\lambda_2=\sqrt{\sqrt7-\frac12}\approx 1.4648.
\]
The eigenvectors corresponding to the eigenvalues \(\pm\lambda_1\) are
\[
e_{\pm}=
\begin{pmatrix}
\mp (\sqrt7+2)\sqrt{1+2\sqrt7},3,-\sqrt2(\sqrt{7}+2),\mp\sqrt{2+4\sqrt7}
\end{pmatrix}^T.
\]
A basis of the center eigenspace is
\[
c_1=
\begin{pmatrix}
0, 2+\sqrt7, \sqrt2, 0
\end{pmatrix}^T,
\qquad
c_2=
\begin{pmatrix}
2-\sqrt7, 0, 0, \sqrt2
\end{pmatrix}^T.
\]
After normalization, we obtain vectors
$E_+,E_-,C_1,C_2$ satisfying
$\omega_0(E_+,E_-)=1$ and $\omega_0(C_1,C_2)=1,$
and all other symplectic pairings equal to zero. We now diagonalize the hyperbolic part. Define
\[
F_1=\frac{-E_++E_-}{\sqrt2},
\qquad
F_2=\frac{-E_+-E_-}{\sqrt2}.
\]
Then $\omega_0(F_1,F_2)=1$ and $V=[F_1, C_1, F_2, C_2]$ becomes a symplectic $4\times 4$ matrix.

Define symplectic coordinates $w=(w_1,w_2,w_3,w_4)$ as
\[
(\hat y_1,\hat y_2,\hat x_1,\hat x_2)^T=
w_1F_1+w_2C_1+w_3F_2+w_4C_2=V\cdot w^T.
\]
Hence, $\omega_0=dw_1\wedge dw_3+dw_2\wedge dw_4.$ In these coordinates, the quadratic Hamiltonian takes the form
\[
Q(w)=\frac{\lambda_1}{2}(w_1^2-w_3^2)
+\frac{\lambda_2}{2}(w_2^2+w_4^2).
\]
Therefore
\[
\widehat K(w)=\h_c+\frac{\lambda_1}{2}(w_1^2-w_3^2)+\frac{\lambda_2}{2}(w_2^2+w_4^2)+O(|w|^3).
\]
The variables $w_1$ and $w_3$ correspond to the hyperbolic directions, while $w_2$ and $w_4$ correspond to the center directions. The Lyapunov orbit of the quadratic model on the level $Q^{-1}(\varepsilon)$ is given by
$$
P_{L,\varepsilon}:=\{w_1=w_3=0,
\quad \lambda_2(w_2^2+w_4^2)
=2\varepsilon\}.
$$

In coordinates $w$, the global radial Liouville vector field $Y_0$ is given by
$$
Y_0(w)=\frac12
\left(w_1\partial_{w_1}+w_3\partial_{w_3}
+w_2\partial_{w_2}+w_4\partial_{w_4}\right)+B,
$$
where $B=\frac12 V^{-1}(0,a^3,-a,0)^T$ is constant.
A direct computation gives
\[
B=d_2\,\partial_{w_2}-d_3\,
\partial_{w_3},
\]
where
\[
d_2=\frac{\sqrt3 (5 \sqrt7 - 13)}{36 (7(2 \sqrt7 - 1))^{1/4}}\approx 0.0047>0,
\qquad
d_3=\frac{(2 \sqrt7 - 1)^{1/4}}{21^{1/4} \cdot 36} (13 + 5 \sqrt7)\approx 0.4899>d_2>0.
\]
In particular, the relevant part of the critical set is contained in $w_3\geq 0$.

We compute
\[
dQ(Y_0)=Q+\lambda_2d_2w_2+\lambda_1d_3w_3.
\]
On the level \(\widehat K=\h_c+\varepsilon>\h_c\), we consider $w_3\in[\sqrt{\varepsilon},\delta_0]$, where $\delta_0>0$ is small and determined below. Using the energy relation, we obtain
\begin{equation}\label{equ: energy relation}
\begin{aligned}
\lambda_1 w_1^2+\lambda_2(w_2^2+w_4^2)
=2\varepsilon+\lambda_1 w_3^2+O(|w|^3)\leq 3(\varepsilon+ w_3^2),
\end{aligned}
\end{equation}
for every $\varepsilon, \delta_0>0$ sufficiently small. Then
\begin{equation}\label{estimates_w}
\begin{aligned}
|w_1| \leq \sqrt{\frac{3}{\lambda_1}(\varepsilon+w_3^2)},\qquad
|w_2|,|w_4| \leq \sqrt{\frac{3}{\lambda_2}(\varepsilon+w_3^2)}.
\end{aligned}
\end{equation}

We see that the last two terms of $dQ(Y_0)$ are of order \(\sqrt{\varepsilon+w_3^2}\), while $Q$ is of order $\varepsilon$. Consequently, the sign of \(dQ(Y_0)\) is not controlled near the saddle-center. For this reason, we introduce another Liouville vector field
$$
Y_1(w):=\frac12\left(3w_1\partial_{w_1}+w_2\partial_{w_2}-w_3\partial_{w_3} +w_4\partial_{w_4}\right).
$$
Since \(Q\) is homogeneous of degree two,
we obtain $dQ(Y_1)=Q(w)+\lambda_1(w_1^2+w_3^2)\geq (\lambda_2/2)|w|^2>0$. Hence, $Y_1$ is transverse to every positive and negative level of the quadratic Hamiltonian.
Moreover, $\widehat K(w)=h_c+Q(w)+O(|w|^3)$, and therefore $d\widehat K(Y_1)=dQ(Y_1)+O(|w|^3)$. Hence, for $\delta_0>0$ sufficiently small, we have
\begin{equation}\label{equ: dK(Y1)}
d\widehat K(Y_1)>\frac{1}{2}dQ(Y_1)>\frac{\lambda_2}{4}|w|^2>0, \qquad \forall |w| <\delta_0.
\end{equation}

We shall therefore use \(Y_1\) near the saddle-center and \(Y_0\) away from the saddle-centers. The next step is to interpolate between these two Liouville vector fields. For every sufficiently small \(\varepsilon>0\), this will produce a Liouville vector field $Y_\varepsilon$
defined near the compact portion of the level $\widehat K^{-1}(\h_c+\varepsilon),$
which coincides with $Y_1$ close to $w_3=0$, coincides with $Y_0$ away from it, and satisfies $d\widehat K(Y_\varepsilon)>0$. This will yield the desired contact form $\alpha_\varepsilon= \iota_{Y_\varepsilon}\omega_0\big|_{\widehat K^{-1}(\h_c+\varepsilon)}.$

We now construct the interpolated Liouville field $Y_\varepsilon$ and prove the transversality. We work in the local coordinates $(w_1,w_2,w_3,w_4)$, where
\[
\widehat K(w)=\h_c+Q(w)+R_3(w),
\qquad
Q(w)=\frac{\lambda_1}{2}(w_1^2-w_3^2)+\frac{\lambda_2}{2}(w_2^2+w_4^2),
\]
and \(R_3(w)=O(|w|^3)\), \(dR_3(w)=O(|w|^2)\). Recall that $Y_0=Y_1-w_1\partial_{w_1}+w_3\partial_{w_3}+B$, where
$$
Y_2:=Y_1-w_1\partial_{w_1}+w_3\partial_{w_3}=\frac12\left(w_1\partial_{w_1}+w_2\partial_{w_2}+w_3\partial_{w_3} +w_4\partial_{w_4}\right),
\qquad B=d_2\partial_{w_2}-d_3\partial_{w_3}.
$$
Thus, $dQ(Y_2)=Q$ and $dQ(B)=\lambda_2d_2w_2+\lambda_1d_3w_3$.

We shall construct the interpolation in the region $w_3>0$. The construction near the other saddle-centers is obtained by symmetry.

Fix $\varepsilon>0$ small and consider the constants $1\leq C<C'_\varepsilon<\delta_0/\sqrt\varepsilon$ to be chosen below. Let $\beta_\varepsilon=\beta_\varepsilon(w_3)$
be a smooth non-decreasing cutoff function satisfying
\[
\beta_\varepsilon(w_3)=0\quad\text{for}\quad w_3\le C\sqrt\varepsilon,
\qquad
\beta_\varepsilon(w_3)=1\quad\text{for}\quad w_3\geq C'_\varepsilon\sqrt\varepsilon.
\]
We define the interpolated primitive by
$
\lambda_\varepsilon=\lambda_1+d(\beta_\varepsilon G),
$
where
$$
\lambda_j=\iota_{Y_j}\omega_0, \ j=0,1, \quad \lambda_0-\lambda_1=dG=-\iota_{X_G}\omega_0, \quad
G(w)=d_3w_1+d_2w_4-w_1w_3.
$$
Notice that $-w_1\partial_{w_1}+w_3\partial_{w_3}+B= -X_G$. Then
$d\lambda_\varepsilon=\omega_0.$
Let \(Y_\varepsilon\) be defined by
$\iota_{Y_\varepsilon}\omega_0 =\lambda_\varepsilon.$

Since
$
\lambda_\varepsilon=(1-\beta_\varepsilon)\lambda_1+\beta_\varepsilon\lambda_0
+G\,d\beta_\varepsilon,
$
we have
\[
Y_\varepsilon=Y_1-\beta_\varepsilon X_G+Z_\varepsilon,
\]
where $\iota_{Z_\varepsilon}\omega_0 =G\,d\beta_\varepsilon.$
Because of \(d\beta_\varepsilon= \beta_\varepsilon'(w_3)\,dw_3\), we get
$Z_\varepsilon=-G(w)X_{\beta_\varepsilon}=G(w)\beta_\varepsilon'(w_3)\partial_{w_1}.$
Therefore
\[
d\widehat K(Y_\varepsilon)=d\widehat K(Y_1)-\beta_\varepsilon(w_3) d\widehat K(X_G)
+G(w)\beta_\varepsilon'(w_3)\partial_{w_1}\widehat K.
\]
We prove that this is positive on the level \(\widehat K^{-1}(\h_c+\varepsilon)\), for every $\varepsilon>0$ sufficiently small, by considering three regions.

First, consider the inner region
$
0\leq w_3\le C\sqrt\varepsilon.
$
Here \(\beta_\varepsilon=0\), hence \(Y_\varepsilon=Y_1\). On the level \(\widehat K=\h_c+\varepsilon\), due to the estimates \eqref{estimates_w} and \eqref{equ: dK(Y1)}, we have
\[
d\widehat K(Y_1)=dQ(Y_1)+dR_3(Y_1)\geq \frac{\lambda_2}{4}|w|^2>0,
\]
for every \(\varepsilon>0\) sufficiently small.

We next consider the outer region
$
C'_\varepsilon\sqrt\varepsilon\le w_3\leq \delta_0,
$
where \(\delta_0>0\) is small. In this region \(\beta_\varepsilon=1\), hence \(Y_\varepsilon=Y_0\). We claim that \(d\widehat K(Y_0)>0\) in this region, for every $\delta_0>0$ fixed sufficiently small, and $\varepsilon>0$ sufficiently small. Here, \(C'_\varepsilon>0\) will be chosen so that $C'_\varepsilon \sqrt{\varepsilon}<\delta_0$.

Indeed,
$
d\widehat K(Y_0)=d\widehat K(Y_1)+d\widehat K(-w_1\partial_{w_1}+w_3\partial_{w_3}+B).
$
The first term is
$
d\widehat K(Y_1)=O(|w|^2).
$
For the second term, we have
\[
d\widehat K(-w_1\partial_{w_1}+w_3\partial_{w_3}+B)
=\lambda_1d_3w_3+\lambda_2d_2w_2+O(|w|^2).
\]
Using the energy relation again,  we obtain, $|w_2|< c_0\left(w_3+\sqrt\varepsilon\right),$ for every $\delta_0>0$ fixed sufficiently small, and $\varepsilon>0$ sufficiently small, where $c_0=\sqrt{3/\lambda_2}$, see \eqref{estimates_w}.
Since \(w_3\geq C'_\varepsilon\sqrt\varepsilon\), we obtain $|w_2|< 2c_0 w_3,$ if $C'_\varepsilon>0$ is sufficiently large, for every $\varepsilon>0$ sufficiently small.
Thus, fixing $\delta_0>C'_\varepsilon\sqrt\varepsilon$ sufficiently small, $C'_\varepsilon>0$ sufficiently large, and using the relation $\lambda_1 d_3 -2\lambda_2 d_2 c_0 \approx 0.8491>0,$
we obtain $d\widehat K(-w_1\partial_{w_1}+w_3\partial_{w_3}+B)\ge c_1 w_3$ for some \(c_1>0\). Hence
$d\widehat K(Y_0)\geq (c_1/2)w_3>0$ on this whole outer region, for every $\varepsilon>0$ sufficiently small.

It remains to treat the transition region
$C\sqrt{\varepsilon}\le w_3\leq C'_\varepsilon\sqrt{\varepsilon}$. In this region the energy relation implies \(|w|=O(C'_\varepsilon\sqrt{\varepsilon})\). Recall that
\[
Y_\varepsilon=Y_1-\beta_\varepsilon X_G+Z_\varepsilon,\qquad -X_G=Y_0-Y_1,\qquad
Z_\varepsilon=G(w)\beta_\varepsilon'(w_3)\partial_{w_1}.
\]
Therefore,
\[
d\widehat K(Y_\varepsilon)=d\widehat K(Y_1)-\beta_\varepsilon\, d\widehat K(X_G)+d\widehat K(Z_\varepsilon).
\]
We now estimate these three terms separately. First, on the level \(\widehat K=\h_c+\varepsilon\), the estimate \eqref{equ: dK(Y1)} implies that
\begin{equation}\label{First_Term}
d\widehat K(Y_1)>\frac{\lambda_2}{4}|w|^2,
\end{equation}
for every $\varepsilon>0$ sufficiently small.

Second, $-d\widehat K(X_G)=-dQ(X_G)-dR_3(X_G).$ Since $-dQ(X_G)=\lambda_1d_3w_3+\lambda_2d_2w_2+O(|w|^2)$, and
 \(dR_3(B)=O(|w|^2)\), we obtain
\[
-d\widehat K(X_G)=\lambda_1d_3w_3+\lambda_2d_2w_2+O(|w|^2).
\]
Choosing any $C\geq 1$, we obtain \(|w_2|\le 2c_0w_3\). Since
$\lambda_1d_3-2\lambda_2d_2c_0 \approx 0.8491>0,$
we obtain
\begin{equation}\label{equ: ineq c2w3}
-d\widehat K(X_G)\geq c_2 w_3>0,\quad c_2:=0.8.
\end{equation}
throughout the transition region, for every $\varepsilon>0$ sufficiently small. Here, $c_2>0$ is a constant independent of $\varepsilon>0$.

Third, $d\widehat K(Z_\varepsilon)=
G(w)\beta_\varepsilon'(w_3)\partial_{w_1}\widehat K.$ Since \(G(w)= d_3 w_1 + d_2 w_4-w_1w_3\) and
\[
\partial_{w_1}\widehat K=\partial_{w_1}Q+\partial_{w_1}R_3=\lambda_1 w_1+O(|w|^2),
\]
together with the estimates in \eqref{estimates_w}, we obtain
\begin{equation}\label{Third_Term}
\begin{aligned}
d\widehat K(Z_\varepsilon)
&= \beta_\varepsilon'(w_3)(d_3 w_1 + d_2 w_4-w_1w_3)(\lambda_1w_1+O(|w|^2))\\
&=\beta_\varepsilon'(w_3)(\lambda_1d_3w_1^2+\lambda_1 d_2 w_1w_4 + O(|w|^3))\\
& \geq \beta_\varepsilon'(w_3)(\lambda_1d_2w_1w_4 +O(|w|^3))\\
& \geq -\beta_\varepsilon'(w_3) c_3(\varepsilon + w_3^2)
\end{aligned}
\end{equation}
for every $\varepsilon>0$ sufficiently small, after decreasing $\delta_0$ if necessary. Here, we take $c_3=4d_2\sqrt{\lambda_1/\lambda_2}\approx 0.0207>0$, which is a constant independent of $\varepsilon$.

Combining \eqref{First_Term}, \eqref{equ: ineq c2w3} and \eqref{Third_Term}, we obtain
$$
d\widehat K(Y_\varepsilon)> \frac{\lambda_2}{4}|w|^2 + \beta_\varepsilon(w_3) c_2w_3 - c_3\beta_\varepsilon'(w_3)(\varepsilon + w_3^2).
$$
We choose \(\beta_\varepsilon=\beta_\varepsilon(w_3), w_3 \in [C\sqrt{\varepsilon},C'_\varepsilon \sqrt{\varepsilon}],\) satisfying the following differential inequality
\begin{equation}\label{differential_inequality}
0\leq \beta_\varepsilon'(w_3)\leq\,\frac{(\lambda_2/4)w_3^2+c_2\beta_\varepsilon(w_3)w_3}{c_3(\varepsilon+w_3^2)} \quad \Rightarrow \quad  d\widehat K(Y_\varepsilon)>0.
\end{equation}
Such a $\beta_\varepsilon$ always exists if $C'_\varepsilon>C=1$ is taken sufficiently large depending on $\varepsilon>0$.

Let $\tilde \beta_\varepsilon(w_3)=\beta_\varepsilon(\sqrt{\varepsilon}w_3),w_3\in[C,C'_\varepsilon]$. Then the differential inequality becomes
$$
\tilde \beta_\varepsilon'(w_3)\leq \frac{\sqrt{\varepsilon}(\lambda_2/4)w_3^2+c_2\tilde \beta_\varepsilon(w_3)w_3}{c_3(1+w_3^2)}.
$$
We consider equality above and initial condition $\tilde \beta_\varepsilon(C)=0$ to obtain the solution
$$
\tilde\beta_\varepsilon(w_3)= \frac{\sqrt{\varepsilon}\lambda_2}{4c_3}(1+w_3^2)^{\frac{c_2}{2c_3}} \int_C^{w_3}w_3^2(1+w_3^2)^{-1-\frac{c_2}{2c_3}}dw_3.
$$
We define $C_\varepsilon'>C$ to satisfy $\tilde \beta_\varepsilon(C_\varepsilon') = 1$. We see that $C_\varepsilon'$ is well-defined since $\tilde \beta_\varepsilon(w_3) \to +\infty$ as $w_3\to +\infty$. Moreover, $C'_\varepsilon\rightarrow +\infty$ as $\varepsilon\rightarrow 0$ since $c_2>c_3$ and thus $K_C:=\int_C^{+\infty}w_3^2(1+w_3^2)^{-1-\frac{c_2}{2c_3}}dw_3<+\infty$.  Since $1/2-c_3/(2c_2)>1/4,$ we conclude that
$$
C_\varepsilon' = O(\varepsilon^{-\frac{c_3}{2c_2}}) \quad \Rightarrow \quad
C'_\varepsilon\sqrt{\varepsilon}=O(\varepsilon^{\frac{1}{2} - \frac{c_3}{2c_2}  })\to 0 \quad \mbox{as } \quad \varepsilon \to 0,
$$

Smoothing $\beta_\varepsilon$ near the values $0$ and $1$ in a standard way, it is then possible to obtain the cut-off function $\beta_\varepsilon:\R \to [0,1]$ satisfying \eqref{differential_inequality} and the other desired properties.  In this way, we find $Y_\varepsilon$ which is transverse to $\widehat K^{-1}(\h_c+ \varepsilon) \cap \{x_3 \geq 0\}$ for every $\varepsilon>0$ sufficiently small.

This proves transversality in the three regions:
$$
0<w_3\leq C\sqrt\varepsilon,\qquad
C\sqrt\varepsilon\leq w_3\leq C'_\varepsilon\sqrt\varepsilon,\qquad
C'_\varepsilon\sqrt\varepsilon\leq w_3<\delta_0.
$$
for every $\varepsilon>0$ sufficiently small.

In the region $w_3>\delta_0$, the Liouville vector field is \(Y_0\), and its transversality follows from the strict convexity of the critical hypersurface away from the saddle-centers. In the region $-C\sqrt{\varepsilon} \leq w_3\leq 0$, $Y_\varepsilon$ coincides with $Y_1$ and its transversality follows from estimates similar to the ones in the region $0\leq w_3\leq C\sqrt{\varepsilon}$. Thus \(Y_\varepsilon\) is transverse to the supercritical level \(\widehat K^{-1}(\h_c+\varepsilon)\)
in the relevant subset of the energy surface.

We conclude this section by describing the finite-energy planes associated with the local saddle-center model. As before, consider the quadratic Hamiltonian $Q(w)=\frac{\lambda_1}{2}(w_1^2-w_3^2)+\frac{\lambda_2}{2}(w_2^2+w_4^2)$ and the Liouville vector field $Y_1=\frac12(3w_1\partial_{w_1}
+w_2\partial_{w_2}-w_3\partial_{w_3}
+w_4\partial_{w_4})$. As observed above, $dQ(Y_1)=Q(w)+\lambda_1(w_1^2+w_3^2)$ is a positive-definite quadratic form, and thus \(Y_1\) is transverse to every regular level \(Q^{-1}(c_0)\), \(c_0>0\).
Let
\[
\alpha_1=\iota_{Y_1} \omega_0=
\frac12w_3\,dw_1-\frac12 w_4\,dw_2
+\frac32 w_1\,dw_3
+\frac12 w_2\,dw_4
\]
be the contact form induced by \(Y_1\) on \(Q^{-1}(c_0)\). A direct computation shows that the contact structure
\(\xi=\ker\alpha_1\) is spanned by the vectors
$$
e_1=X_1-\frac{dQ(X_1)}{dQ(Y_1)}Y_1,
\qquad
e_2=X_2-\frac{dQ(X_2)}{dQ(Y_1)}Y_1,
$$
where
\[
X_1=
\frac12 w_4\partial_{w_1}
+\frac12 w_3\partial_{w_2}
+\frac12 w_2\partial_{w_3}
-\frac32 w_1\partial_{w_4},
\quad
X_2=
-\frac12 w_2\partial_{w_1}
+\frac32 w_1\partial_{w_2}
+\frac12 w_4\partial_{w_3}
+\frac12w_3\partial_{w_4}.
\]

We choose the compatible almost complex structure \(J\) on
\(\R\times Q^{-1}(c_0)\) determined by
$
J\cdot e_1=e_2$ and $J\cdot \partial_a=R_1$, where $R_1=X_Q/\alpha_1(X_Q)$ is the Reeb vector field of $\alpha_1$.
Consider a \(J\)-holomorphic cylinder
$\tilde u=(a,u):
\R\times \R/\Z
\to
\R\times Q^{-1}(c_0)
$
satisfying
\[
\pi\partial_tu=J\cdot \pi\partial_su,
\qquad
\partial_sa=\alpha_1(\partial_tu),
\qquad
\partial_ta=-\alpha_1(\partial_su),
\]
where $\pi:TQ^{-1}(c_0) \to \xi$ is the projection along the Reeb vector field $R_1$.

We look for solutions contained in the invariant sphere
\(\{w_3=0\}\) of the form
\[
u(s,t)=\bigl(w_1(s),r(s)\cos(2\pi t),0,
r(s)\sin(2\pi t)\bigr),
\]
where $w_1(s)=\pm \sqrt{(2c_0-\lambda_2r(s)^2)/\lambda_1}$ and $r\geq 0$. Substituting the above ansatz into the Cauchy-Riemann equations yields $r'(s)=g(r),$ where
\[
g(r)=\frac{2\pi r(2c_0-\lambda_2r^2)
\bigl(\lambda_1r^2+9(2c_0-\lambda_2r^2)
\bigr)}{\bigl(2c_0+2(2c_0-\lambda_2r^2)
\bigr)^2}.
\]

Setting
$r_0=\sqrt{2c_0/\lambda_2},$ we have $g(r)>0$ for $0<r<r_0,$
 and $g(0)=g(r_0)=0$. Therefore, every solution with
\(r(0)\in(0,r_0)\) satisfies
$r(s)\to 0$ as $s\to-\infty,$ and $r(s)\to r_0$ as $s\to+\infty.$

Choose $w_1(s)=\sqrt{(2c_0-\lambda_2r(s)^2)/\lambda_1}.$ The positive end is asymptotic to the Lyapunov orbit
$P_{L,c_0}=Q^{-1}(c_0)\cap\{w_1=w_3=0\},$
while the negative end converges to the  point $(\sqrt{2c_0/\lambda_1},0,0,0).$
Hence, the negative puncture is removable, and the cylinder extends to an embedded finite-energy $J$-holomorphic plane $\tilde u_+:\C\to\R\times Q^{-1}(c_0).$
Similarly, choosing $w_1(s)=-\sqrt{(2c_0-\lambda_2r(s)^2)/\lambda_1}$, gives a second embedded finite-energy plane
$\tilde u_-:\C\to\R\times Q^{-1}(c_0).$
The projections of these two planes are embedded, disjoint, and coincide with the hemispheres of the two-sphere $
Q^{-1}(c_0)\cap\{w_3=0\}.$

We now transfer these local planes to the original Hamiltonian energy
surfaces slightly above the critical value. Recall that, near the saddle-center $(0,a^3,-a,0)$, the rescaled Hamiltonian has the form
$\widehat K_\varepsilon(w)=
\varepsilon^{-1}\bigl(\widehat K(\sqrt{\varepsilon}\,w)-\h_c\bigr),$
and satisfies $
\widehat K_\varepsilon \to Q$ in $C^\infty_{\mathrm{loc}}$
as \(\varepsilon\to0^+\). Notice that $\widehat K_\varepsilon$ can be regarded as a smooth family of functions depending on a parameter $c$, which is defined near $c=0$, and satisfies $c=\sqrt{\varepsilon}$ for every $\varepsilon>0$ sufficiently small.

The rescaling map $
\Phi_\varepsilon(w):=\sqrt{\varepsilon}\,w$ identifies compact subsets of the model level \(Q^{-1}(1)\) with subsets of the original energy level
$
\widehat K^{-1}(\h_c+\varepsilon)$
near the saddle-center. Due to the construction of $Y_\varepsilon$, we can assume that the contact form $\alpha_\varepsilon$ coincides with $\alpha_1$ in the uniform region $-1\leq w_3\leq 1$.

By the persistence of nondegenerate hyperbolic periodic orbits, for every \(\varepsilon>0\) sufficiently small, the Lyapunov orbit \(P_{L,1}\subset Q^{-1}(1)\) gives rise to a hyperbolic periodic orbit
$\widehat P_{2,\varepsilon}\subset
\widehat K^{-1}(\h_c+\varepsilon)$
satisfying $\varepsilon^{-1/2}\widehat P_{2,\varepsilon}\to P_{L,1}$ in $C^\infty.$

Consider the contact form $\alpha_\varepsilon = \alpha_1 |_{\widehat K_\varepsilon^{-1}(1)}$ on $\widehat K_\varepsilon^{-1}(1)$ for every $\varepsilon>0$ sufficiently small. Take any $\alpha_\varepsilon$-adapted almost complex structure $\widehat J_\varepsilon$ in the symplectization of $\widehat K_\varepsilon^{-1}(1)$, smoothly depending on $c$ near $c=0$, with $c=\sqrt{\varepsilon}$ for $\varepsilon>0$ small. Since the finite-energy planes \(\tilde u_\pm\) are automatically Fredholm regular \cite{Wendl10a}, they persist for every $\varepsilon>0$ sufficiently small.  Thus, for every sufficiently small \(\varepsilon>0\), there exist finite-energy planes
$$
\tilde u_{\pm,\varepsilon}
=
(a_{\pm,\varepsilon},u_{\pm,\varepsilon})
:
\mathbb C
\to
\mathbb R\times \widehat K_\varepsilon^{-1}(1)
$$
asymptotic to the corresponding Lyapunov orbit in the rescaled level. The projections of these planes are located in $-1/2\leq w_3\leq 1/2$ for every $\varepsilon>0$ sufficiently small.

Returning to the original coordinates, define
$$
\tilde U_{\pm,\varepsilon}
=(A_{\pm,\varepsilon},U_{\pm,\varepsilon}):\mathbb C\to\mathbb R\times \widehat K^{-1}(\h_c+\varepsilon)
$$
by $A_{\pm,\varepsilon}:=\varepsilon\, a_{\pm,\varepsilon}$ and $
U_{\pm,\varepsilon}:= \Phi_\varepsilon \circ u_{\pm,\varepsilon}.$ The almost complex structure on
\(\mathbb R\times\widehat K^{-1}(\h_c+\varepsilon)\), still denoted $\hat J_\varepsilon$, is the one induced by the rescaling map and by the contact form
$\alpha_\varepsilon=
\iota_{Y_\varepsilon}\omega_0
\big|_{\widehat K^{-1}(\h_c+\varepsilon)}.$

With this choice, the maps \(\tilde U_{\pm,\varepsilon}\) are
\(\widehat J_\varepsilon\)-holomorphic finite-energy planes asymptotic to
\(\widehat P_{2,\varepsilon}\). Their projections lie in an \(O(\sqrt{\varepsilon})\)-neighborhood of the saddle-center and converge, after applying the inverse rescaling, to the model planes in \(Q^{-1}(1)\). In particular, the two planes approach the Lyapunov orbit through opposite directions and their projections form, together with \(\widehat P_{2,\varepsilon}\), the two hemispheres of the transverse sphere in the neck region.

\section{Proof of Proposition \ref{prop_retrograde}}\label{sec_retrograde}

We briefly recall Birkhoff's shooting argument. Following \cite{Birkhoff1915}, see also \cite[Chapter~8]{FvK}, we consider the family of initial conditions
\[
(q_1(0),q_2(0),\dot q_1(0),\dot q_2(0))
=(q_{1,0},0,0,\dot q_{2,0}), \qquad
q_{1,0}<0,\ \dot q_{2,0}<0,
\]
on the energy surface \(H^{-1}(E)\), where $E<E_{c}$.

Using the equations of motion \eqref{equ: 2nd ODE}, one shows that if \(q_1(0)<0\) is sufficiently close to \(0\), then the corresponding solution remains in the lower half-plane and reaches the axis \(q_1=0\) at some time
\(t_0>0\). Moreover,
$\arg(\dot q(t_0))\to-\pi/4$ as
$q_{1,0}\to 0^-$.

On the other hand, let \(\bar q_1>0\) be uniquely determined by $U(-\bar q_1,0)=E$. If \(q_1(0)=-\bar q_1\), then
\(\dot q_1(0)=\dot q_2(0)=0\), and a direct computation from \eqref{equ: 2nd ODE} gives $\ddot q_2(0)=0,$ and $\dddot q_2(0)>0.$ Hence, the corresponding trajectory immediately enters the upper half-plane.

By continuity, trajectories with an initial condition
\(q_1(0)\in(-\bar q_1,0)\) may have both possible behaviors.
For such an initial condition, let \(t_0>0\) denote the first time at which
either \(q_1(t_0)=0\) or \(q_2(t_0)=0\). In particular,
$
q_1(t),q_2(t)<0,
$ and $
0<t<t_0.$

It is straightforward to check that the Hill region
\(B_E\) is contained in the unit disk for every \(E<E_c\). Using
\eqref{equ: 2nd ODE}, we see that
$
(\dot q_1+2q_2)(0)=0,
$
and
\[
\frac{d}{dt}(\dot q_1+2q_2)
=
-\partial_{q_1}U
=
3q_1\Bigl(1-\frac1{|q|^3}\Bigr)
>0
\]
as $q(t)$ enters the lower half-plane. Therefore, we obtain that $(\dot q_1+2q_2)(t)>0$ for every \(t\in(0,t_0)\). Consequently,
$\dot q_1(t)>-2q_2(t)>0.$
In particular, \(q_1(t)\) increases monotonically and \(t_0>0\) is finite.

Let \(q_{1,*}<0\) be the smallest number so that all solutions with initial conditions satisfying $q_{1,*}<q_{1,0}<0$ remain in the third quadrant $q_1,q_2<0$ for $0<t<t_0$, and hits $q_1=0, q_2<0$ transversely at time $t=t_0$. By the
previous discussion, \(q_{1,*}>-\bar q_1\). Standard shooting arguments show
that the limiting trajectory corresponding to \(q_{1,*}\) necessarily ends in
collision. Consequently, for initial conditions sufficiently close to
\(q_{1,*}\), the trajectory crosses the symmetry axis \(q_1=0\) transversely.

By continuity, there exists an intermediate initial condition for which the
trajectory meets the symmetry axis $q_1=0$ orthogonally. Reflecting successively with
respect to the coordinate axes produces a symmetric periodic orbit. This is the so-called Birkhoff
retrograde orbit.

We now extend Birkhoff's shooting argument to energies equal or slightly above
the critical value \(E_c\). Not only we need to prove the existence of the retrograde orbit, but we also need to show that there exists a continuous family of such orbits for $E$ sufficiently close to the critical level. Hence, consider the two families of initial conditions
\[
\{(0,p_2,q_1,0)\mid -1\le q_1<0,\ p_2+q_1<0\}
\quad \mbox{ and } \quad
\{(0,\hat p_2,\hat q_1,0)\mid 0<\hat q_1\le 1,\ \hat p_2+\hat q_1>0\}
\]
on the energy surface \(H^{-1}(E)\). For every initial condition sufficiently
close to $q=(0,0)$, let \(t_0>0\) denote the first time such that either
\(q_1(t_0)=0\) or \(\hat q_1(-t_0)=0\).
The corresponding solutions define two real-analytic shooting curves in the
rectangle $
Q:=[-\pi/2,\pi/2]\times[-\bar q_2,0],$
namely
\[
\Gamma_1=\{(\arg(\dot q(t_0)),q_2(t_0))
:q_1(0)\in(-\varepsilon_0,0)\},
\quad
\Gamma_2=\{(\arg(\dot{\hat q}(-t_0)),\hat q_2(-t_0)):\hat q_1(0)\in(0,\varepsilon_0)\},
\]
where \(\varepsilon_0>0\) is sufficiently small and \(\bar q_2>0\) is defined by $U(0,\bar q_2)=E.$ By the reversibility of the equations of motion, the curves \(\Gamma_1\) and \(\Gamma_2\) are symmetric with respect to the vertical axis \(\{0\}\times\mathbb R\). Furthermore,
$\Gamma_1\to(-\pi/4,0)$ as $q_1(0)\to 0^-$, and $\Gamma_2\to(\pi/4,0)$ as $
\hat q_1(0)\to0^+.$

To continue the shooting curves towards the equilibrium point \(S_1\), we need the following lemma.

\begin{lem}\label{lem: hat q2(0)>-1+delta}
There exists \(\delta_0>0\) small such that for every energy \(E\) sufficiently close
to \(E_c\), the solution satisfying
$q_1(0)=-1+\delta_0,$ $q_2(0)=0,$ $\dot q_1(0)=0$ and $\dot q_2(0)<0,$ admits a time \(t_0>0\) such that $q_2(t_0)=0,$
while $q_2(t)<0<\dot q_1(t)$ for every \(t\in(0,t_0)\). Moreover, $q_1(0)< q_1(t_0)<0$ and $\dot q_2(t_0)>0.$
\end{lem}

\begin{proof}
Let $\bar x=(p-i,q+1)(W^T)^{-1}$ be the local symplectic coordinates, so that the equilibrium point \(S_1=(i,-1)\) corresponds to \(\bar x=0\). Here, $W$ is the $4\times 4$ symplectic matrix
\begin{equation} \label{equ: V}
W:=\begin{pmatrix}
\frac{C_2^2+2}{\sqrt{\lambda_1'}C_1C_0} & 0 & 0 & \frac{C_1^2+2}{\sqrt{\lambda_2'}C_2C_0}\\
0 & \frac{\sqrt{\lambda_2'}(C_2^2-2)}{C_2C_0} & \frac{\sqrt{\lambda_1'}(C_1^2-2)}{C_1C_0} & 0\\
0 & \frac{2\sqrt{\lambda_2'}}{C_2C_0} & \frac{2\sqrt{\lambda_1'}}{C_1C_0} & 0\\
\frac{C_1}{\sqrt{\lambda_1'}C_0}&0&0&\frac{C_2}{\sqrt{\lambda_2'}C_0}
\end{pmatrix}.
\end{equation}
where $C_0=2\times 7^{1/4}, C_1=\sqrt{7}-1, C_2=\sqrt{7}+1$, $\lambda_1'=\sqrt{1+2\sqrt{7}}$ and $\lambda_2'=\sqrt{-1+2\sqrt{7}}$.
For \(E=E_c\), the Hamiltonian takes the form $H=E_c+H_2(\bar p,\bar q)+O(|(p,q)|^3)=E_c+H_2(\bar x)+O(|\bar x|^3),$ where
\[
H_2(\bar x)=\frac{(\bar p_1-\bar q_2)^2}{2}+\frac{(\bar p_2+\bar q_1)^2}{2}-9\bar q_1^2+3\bar q_2^2=\frac{\lambda_1'}{2}(\bar x_1^2-\bar x_3^2) + \frac{\lambda_2'}{2}(\bar x_2^2+\bar x_4^2).
\]

The initial condition $q_1(0)=-1+\delta_0,$ $q_2(0)=0$, $\dot q_1(0)=0$ and $\dot q_2(0)<0$ corresponds, after rescaling \(\bar x=\delta_0 x\), to an initial condition converging to $x(0) = \kappa (0, -1/\sqrt{\lambda_2'}, 1/\sqrt{\lambda_1'}, 0 ) \in H_2^{-1}(0),$ where \(\kappa=\frac{C_0C_1C_2}{2(C_2-C_1)}>0\). The linearized equations at $S_1$ are
\[
\dot x_1=\lambda_1'x_3,\qquad \dot x_3=\lambda_1'x_1, \qquad \dot x_2=-\lambda_2'x_4,\qquad \dot x_4=\lambda_2'x_2.
\]
Thus, the limiting trajectory is \[ x(t) = \kappa \left( \frac{\sinh(\lambda_1't)}{\sqrt{\lambda_1'}}, -\frac{\cos(\lambda_2't)}{\sqrt{\lambda_2'}}, \frac{\cosh(\lambda_1't)}{\sqrt{\lambda_1'}}, -\frac{\sin(\lambda_2't)}{\sqrt{\lambda_2'}} \right). \]

The relation between the \(x\)-coordinates and the  \(q\)-coordinates is linear. Namely, if $(p-i,q+1)=\bar x\, W^T,$ then, after the rescaling \(\bar x=\delta_0x\), the displacement of \(q\) from \((-1,0)\) is $ q+1=\delta_0\,L(x\, W^T),$ where \(L\) is the projection from $(p,q)$ onto the \(q\)-coordinates. Therefore, for the limiting trajectory, we may write $q(t)+1=\delta_0\,\tilde q(t),$ where $\tilde q(t)=L(x\, W^T).$ Using the explicit form of \(W\), this gives
\[
\tilde q(t) = \frac{\kappa}{C_0} \left( -\frac{2\cos(\lambda_2't)}{C_2} + \frac{2\cosh(\lambda_1't)}{C_1}, \, \frac{C_1\sinh(\lambda_1't)}{\lambda_1'} - \frac{C_2\sin(\lambda_2't)}{\lambda_2'} \right),
\]
In particular, $\tilde q(0)=(\tilde q_1(0),0)\, =(1,0),\, \dot{\tilde q}_1(0)=0$ and $\dot{\tilde q}_2(0)\,=-\frac{C_1C_2}{2}<0.$ Moreover, the function $\tilde q_2(t)$ is negative for \(t>0\) small and has a first positive zero \(t_0\in(0,\pi/\lambda_2')\). At this zero, $\dot{\tilde q}_2(t_0)>0$ and $\tilde q_1(t_0)>\tilde q_1(0).$ Indeed, the \(\cosh(\lambda_1't)\)-term increases and the \(-\cos(\lambda_2't)\)-term is also larger at the first zero \(t_0>0\) than at \(t=0\).

Since the rescaled nonlinear Hamiltonians converge in \(C^\infty_{\mathrm{loc}}\) to \(H_2\), the corresponding nonlinear trajectories converge in \(C^\infty_{\mathrm{loc}}\) to the trajectory above. Hence, for every \(\delta_0>0\) sufficiently small, the original trajectory also admits $t_0>0$ satisfying $ q_2(t)<0$ for $0<t<t_0,$ $q_2(t_0)=0,$ $\dot q_2(t_0)>0,$ and $q_1(0)<q_1(t_0)<0$. Finally, the monotonicity \(\dot q_1(t)>0\) on \((0,t_0)\) follows from \((\dot q_1+2q_2)(0)=0\) and $\frac{d}{dt}(\dot q_1+2q_2) = -\partial_{q_1}U >0$ in the Hill region that $q_1<0,|q|<1$, for a fixed $\delta_0>0$ and every $|E-E_c|$ sufficiently small. Since \(q_2(t)<0\), this gives $\dot q_1(t)>-2q_2(t)>0.$ The lemma follows.
\end{proof}

Now let $-1<q_{1,*}<0$ be the least value so that the solution with initial conditions satisfying $q_{1,*}<q_{1,0}<0$ remains in the third quadrant $q_1,q_2<0$ for $0<t<t_0$ and reaches the symmetry axis \(q_1=0\) transversely at $t=t_0$. By Lemma~\ref{lem: hat q2(0)>-1+delta}, we have $q_{1,*}>-1+\delta_0$ for every energy sufficiently close to \(E_c\). It can be proved as in \cite{Birkhoff1915} that the solution with $q_{1,0}=q_{1,*}$ admits a collision, that is $q(t_0)=0$, and the limiting argument satisfies $\arg(\dot q(t_0))>\pi/4$.

Consequently, the shooting curve \(\Gamma_1\) extends as a real-analytic curve from the point \((-\pi/4,0)\) to a limiting point \((\theta_*,0)\) with \(\theta_*>\pi/4\). Since \(\Gamma_2\) is obtained from \(\Gamma_1\) by reflection with respect to the vertical axis, the two curves necessarily cross each other at the segment $\{0\}\times(-\bar q_2,0).$
Every such intersection corresponds to a trajectory meeting the symmetry axis orthogonally. Successive reflections with respect to the coordinate axes produce a symmetric retrograde periodic orbit. In particular, there is a topological crossing that persists for all energies sufficiently close to \(E_c\), giving the continuous family of retrograde orbits for $E$ sufficiently close to $E_c$.

Finally, since the shooting curves depend continuously on the energy, the
corresponding retrograde orbits converge in \(C^\infty\) to the critical
retrograde orbit as \(E\to E_c\). Since the retrograde orbits stay away from the singularities, the family of $2$-disks for such retrograde orbits can be chosen continuously and thus with uniformly bounded $|d\alpha|$-area also follows.

\section{Proof of Proposition \ref{prop_weakly_convex}}\label{sec_weakly_convex}


We shall use the following index estimate, which is essentially proved in
\cite[Theorem~1.9]{LS25} using the Robbin-Salamon index and holds for a broader class of mechanical-magnetic Hamiltonians.

\begin{lem}\label{lem: RS index estimate}
There exists an open neighborhood
\(\mathcal U_3\subset\mathbb R^4\) of the singularities corresponding to
\(S_1\) and \(S_2\) such that if $P$ is a contractible periodic orbit of $K$ satisfying $P\cap\mathcal U_3\neq\emptyset$ and \(P\) is not one of the Lyapunov orbits, then
$\mu_{\mathrm{CZ}}(P)\ge 3.$
\end{lem}

Lemma~\ref{lem: RS index estimate} implies that every contractible periodic orbit intersecting $\mathcal U_3$ admits index at least $3$, except for the Lyapunov orbits $P_{2,1,E}, P_{2,2,E}$.  Moreover, due to the strict convexity of $\mathcal M_{E_c}$ proved in Proposition \ref{prop_strict_convexity}, the Conley-Zehnder indices grow linearly with the actions, in the sense that there exist constants  $A, B$, with $B>0$, so that if a periodic orbit in $\mathcal M_{E_c}$ has action $T>0$ then $\mu_{\rm CZ}(P) \geq A+B T.$ The constants $A,B$ can be estimated using the curvatures of $\mathcal M_{E_c}$ which are strictly positive. Hence, we can find uniform constants $A,B$, with $B>0$, so that the same relation holds for every periodic orbit in $\mathcal M_E\setminus \mathcal U_3$, for every $E-E_c>0$ sufficiently small.

Now suppose by contradiction that there exists $E_k \to E_c^+$ as $k\to \infty$ so that $\mathcal M_{E_k}$ admits a periodic orbit $P_k$ whose Conley-Zehnder index is $\leq 2$ for every $k$, and $P_k$ is not a Lyapunov orbit. By Lemma \ref{lem: RS index estimate}, $P_k \subset \mathcal M_{E_k} \setminus \mathcal U_3$ for $k$ sufficiently large. From the action-index relation, we conclude that the actions of $P_k$ are uniformly bounded. Thus, we can pass to the limit $k\to \infty$ and obtain a periodic orbit $P_c \subset \mathcal M_{E_c}\setminus \mathcal U_3$ so that $P_k \to P_{c}$ in $C^\infty$ as $k\to \infty$. Since the Conley-Zehnder index is lower semi-continuous, we conclude that $\mu_{\rm CZ}(P_c) \leq 2$, a contradiction. This implies that $\mathcal M_E$ is weakly convex for every $E-E_c>0$ sufficiently small and the only periodic orbits with index $2$ are the Lyapunov orbits near the equilibrium points.


\begin{thebibliography}{10}
\bibitem{AFKP12}
P. Albers, U. Frauenfelder, O. van Koert and G. P. Paternain, {\it Contact Geometry of the Restricted there-body problem},
\newblock  Communication on Pure and Applied Mathematics,
LXV(2012), pp: 0229-0263.

\bibitem{AFFHvK} P. Albers, J. Fish, U. Frauenfelder, H. Hofer and O. van Koert, {\it Global surfaces of section in the planar restricted $3$-body problem.}  Archive for Rational Mechanics and Analysis, 204(2012), 273--284.

\bibitem{Aydin1} C. Aydin,
{\it The Conley-Zehnder indices of the spatial Hill three-body problem},
Celestial Mechanics and Dynamical Astronomy 135 (2023), Article 29.

\bibitem{Aydin2} C. Aydin and A. Batkhin,
{\it Studying network of symmetric periodic orbit families of the Hill problem via symplectic invariants},
arXiv:2410.21245, 2024.


\bibitem{BanLong2010}
V. Bangert, Y. Long,
\newblock {\it The existence of two closed geodesics on every Finsler $2$-sphere},
\newblock{\it  Math. Ann.} 346(2010), 335-366.

\bibitem{Birkhoff1915}
G. Birkhoff,
\newblock {\it The restricted problem of three bodies},
\newblock{Rend. Circ. Matem. Palermo}, 39(1915), 265¨C334.

\bibitem{Brown1896}
E. Brown,
\newblock{\it An introductory treatise on the lunar theory},
\newblock Cambridge University Press (1896).

\bibitem{Conley1968} C. Conley. {\it Twist mappings, linking, analyticity, and periodic solutions which pass close to an unstable periodic solution.} Topological dynamics, (1968), 129--153.

\bibitem{dPS1} N. V. de Paulo and Pedro A. S. Salom\~ao, {\it  Systems of transversal sections near critical energy levels of Hamiltonian systems in $\R^4$}, Memoirs of the Amer. Math. Soc., 252(2018), no. 1202, 1--105.

\bibitem{dPS2} N. V. de Paulo and Pedro A. S. Salom\~ao, {\it On the multiplicity of periodic orbits and homoclinics near critical energy levels of Hamiltonian systems in $\mathbb R^4$}, Transactions of the Amer. Math. Soc., 372(2019), no. 2, 859--887.

\bibitem{dPS3} N. V. de Paulo and Pedro A. S. Salom\~ao. {\it Reeb flows, pseudo-holomorphic curves and transverse foliations.} S\~ao Paulo Journal of Mathematical Sciences 16, no. 1 (2022): 314--339.

\bibitem{dPHKS} N. de Paulo, U. Hryniewicz, S. Kim and Pedro A. S. Salom\~ao. {\it Genus zero transverse foliations for weakly convex Reeb flows on the tight $3$-sphere}, to appear in Advances in Mathematics.




\bibitem{FS} J. W. Fish and R. Siefring. {\it Connected sums and finite energy foliations I:
Contact connected sums}, J. Symplectic Geom. 16(2018), no. 6, 1639--1748

\bibitem{franks1} J. Franks. {\it Generalizations of the Poincar\'e-Birkhoff theorem.} Ann. of Math. 128(1988), 139--151.

\bibitem{franks2} J. Franks. {\it Erratum to ``Generalizations of the Poincar\'e-Birkhoff theorem''.} Ann. of Math. 164(2006), 1097--1098.

\bibitem{FvK} U. Frauenfelder and O. van Koert.
\newblock {\it The restricted three-body problem and holomorphic curves}.
\newblock Springer International Publishing, 2018.

\bibitem{FKZ2016}
U. Frauenfelder, O. van Koert and L. Zhao,
\newblock {\it On the problem of convexity for the restricted three-body problem around the heavy primary},
\newblock{\it Hokkaido Math.J.} 51(2022), no. 2, 287--317.

\bibitem{GRS} C. Grotta-Ragazzo, and Pedro A. S. Salom\~ao. {\it The Conley¨CZehnder index and the saddle-center equilibrium.} Journal of Differential Equations 220(2006), no. 1, 259--278.

\bibitem{C-GLS2025}
C. Grotta-Ragazzo, L. Liu and
Pedro A. S. Salom\~ao,
\newblock {\it Non-resonant Hopf links near a Hamiltonian equilibrium point}, to appear in Communications in Mathematical Physics.



\bibitem{Hill1}
G. Hill,
\newblock {\it Researches in the lunar theory},
\newblock Amer. J. Math., 1(1878), 5-26, 129-147, 245-260.

\bibitem{Hill2}
G. Hill,
\newblock {\it On the part of the Motion of the Lunar Perigee which is a Function of the mean motions of the Sun and the Moon,}
\newblock{John Wilson \& Son, Cambridge, Massachusetts,(1877), Reprinted in Acta} 8(1886), 1-36.

\bibitem{Hofer1993}
H.~Hofer,
\newblock {\it Pseudoholomorphic curves in symplectizations with applications to the Weinstein conjecture in dimension three},
\newblock  Invent. Math., 114(1993), 515--563.

\bibitem{char1}
H. Hofer, K. Wysocki and E. Zehnder. \textit{\it A characterization of the tight three sphere.} Duke Math. J. 81 (1995), no. 1, 159--226.

\bibitem{char2}
H. Hofer, K. Wysocki and E. Zehnder. \textit{\it A characterization of the tight three sphere II.} Commun. Pure Appl. Math. 55(1999), no. 9, 1139--1177.

\bibitem{HWZI}
H.~Hofer, K.~Wysocki, E.~Zehnder,
\newblock {\it Properties of pseudoholomorphic curves in symplectictisation I: Asymptotics},
\newblock Ann. Inst. Henri Poincar\'e, 13(1996), 337--379.

\bibitem{HWZII}
H.~Hofer, K.~Wysocki, E.~Zehnder,
\newblock {\it Properties of pseudoholomorphic curves in symplectictisation II: Embedding control and algebraic invariants},
\newblock  Geom. and Funct. Anal., 5(1995), no. 2, 270--328.

\bibitem{HWZIII}H. Hofer, K. Wysocki and E. Zehnder. \textit{Properties of pseudoholomorphic curves in symplectizations III: Fredholm theory.} Topics in nonlinear analysis, Birkh\"auser, Basel, (1999), 381--475.

\bibitem{HWZ98}
H. Hofer, K. Wysocki and E. Zehnder.
\newblock {\it The dynamics of strictly convex energy surfaces in $\mathbb R^4$}.
\newblock Ann. of Math., 148(1998), 197--289.

\bibitem{HWZ03}
H. Hofer, K. Wysocki and E. Zehnder. \newblock {\it Finite energy foliations of tight three-spheres and Hamiltonian dynamics}.
\newblock Ann. of Math, 157(2003), 125--255.

\bibitem{hryn2}U. Hryniewicz. {\it Systems of global surfaces of section for dynamically convex Reeb flows on the $3$-sphere,}  Journal of Symplectic Geometry, 12(2014), no. 4, 791--862.

\bibitem{HLS2014} U. Hryniewicz, J. Licata and Pedro A. S. Salom\~ao. {\it A dynamical characterization of universally tight lens spaces}, Proceedings of the London Mathematical Society, 110(2014),  213--269.

\bibitem{HS1} U. Hryniewicz and Pedro A. S. Salom\~ao. {\it On the existence of disk-like global sections for Reeb flows on the tight $3$-sphere.} Duke Mathematical Journal, 160(2011), no. 3, 415--465.

\bibitem{HS2} U. Hryniewicz and Pedro A. S. Salom\~ao. {\it Elliptic bindings for dynamically convex Reeb flows on the real projective three-space.} Calc. Var. Partial Differential Equations 55 (2016), no. 2, Art. 43, 57 pp.


\bibitem{HSW2023} U. Hryniewicz,  Pedro A. S. Salom\~ao, and K. Wysocki.
\newblock {\it Genus zero global surfaces of section for Reeb flows and a result of Birkhoff}.
\newblock Journal of the European Mathematical Society, 25 (9), 2023.

\bibitem{HLOSY} X. Hu, L. Liu, Y. Ou, Pedro A. S. Salom\~ao, G. Yu. {\it A symplectic dynamics approach to the spatial isosceles three-body problem}, to appear in Journal of the European Mathematical Society.

\bibitem{JvK} Joung, Chankyu, and Otto van Koert. {\it Computational symplectic topology and symmetric orbits in the restricted three-body problem.} Nonlinearity 38(2025), no. 2, 025015.

\bibitem{Kim1} S. Kim. {\it On a convex embedding of the Euler problem of two fixed centers.} Regular and Chaotic Dynamics 23 (2018), 304--324.


\bibitem{Lee2017}
J. Lee,
\newblock {\it Fiberwise Convexity of Hill's lunar problem},
\newblock{J. Topol. Anal.}, 9(2017), no.4, 571-630.

\bibitem{Ligon} T. Ligon, {\it Hill's Lunar Equations, Series, Convergence, Motion of the Perigee}, arXiv: 2512.08961, 2025.

\bibitem{LS25}
L. Liu, Pedro A. S. Salom\~ao,
\newblock {\it Finite energy foliations and global dynamics in the restricted three-body problem},
\newblock{arXiv:2506.17867}


\bibitem{Long02}
Y.~Long.
\newblock {\it Index theory for symplectic paths with applications}, volume 207,
\newblock  Progress in Mathematics,
Birkh\"{a}user Verlag, Basel, 2002.

\bibitem{McGehee1969}
R. McGehee.
\newblock {\it Some homoclinic orbits for the restricted three-body problem},
\newblock The University of Wisconsin, Ph.D. Thesis, 1969.

\bibitem{MIW} E. Meletlidou,  S. Ichtiaroglou, and F. J. Winterberg. {\it Non¡ªintegrability of Hill's lunar problem.} Celestial Mechanics and Dynamical Astronomy 80.2 (2001): 145-156.

\bibitem{M-RSS} J. Morales-Ruiz,  C. Sim¨®, and S. Simon. {\it Algebraic proof of the non-integrability of Hill's problem.} Ergodic Theory and Dynamical Systems 25.4 (2005): 1237--1256.



\bibitem{Po}
H. Poincar\'e.
\textit{Sur un th\'eor\`eme de g\'eom\'etrie.} Rend. Circ. Mat. Palermo, 33(1912), 375--407.

\bibitem{Salomao2004}
Pedro A. S. Salom\~ao,
\newblock {\it Convex energy levels of Hamiltonian systems}.
\newblock {\it Qual. Theory Dyn. Syst.}, 4(2), 439-457, 2004.

\bibitem{SH18}
Pedro A.~S. Salom\~ao and U.~L. Hryniewicz.
\newblock {\it Global surfaces of section for {R}eeb flows in dimension three and beyond}.
\newblock In  Proceedings of the International Congress of Mathematicians---Rio de Janeiro 2018. Vol. II. Invited lectures, pages 941--967. World Sci. Publ., Hackensack, NJ, 2018.

\bibitem{Sch} A. Schneider. {\it Global surfaces of section for dynamically convex Reeb flows on lens spaces,}  Transactions of the American Mathematical Society, 373(2020), no. 4, 2775--2803.

\bibitem{SiSt2000} C. Sim\'o and  T.J. Stuchi,
\newblock {\it Central stable/unstable manifolds and the destruction of KAM tori in the planar Hill problem},
\newblock Physica D., {\bf 140}, 1-32, 2000.

\bibitem{Thorpe} John A. Thorpe, \textit{Elementary topics in differential geometry}. Springer Science $\&$ Business Media, 2012.



\bibitem{Wendl08}
C.~Wendl,
\newblock {\it Finite energy foliations and surgery on transverse links}, Ph.D. Thesis,
\newblock  New York University, 2005.

\bibitem{Wendl08b} C. Wendl, {\it Finite energy foliations on overtwisted contact manifolds}.  Geometry \& Topology , 12(2008), no. 1, 531--616.

\bibitem{Wendl10a} C. Wendl. {\it Automatic transversality and orbifolds of punctured holomorphic curves in dimension four}.  Commentarii Mathematici Helvetici, 85(2010), no. 2, 347-407.

\bibitem{Wendl10} C. Wendl, {\it Compactness for embedded pseudoholomorphic curves in $3$-manifolds}. Journal of the European Mathematical Society, 12(2010), no. 2, 313--342.

\bibitem{Wintner1}
A. Wintner,
\newblock {\it \"Uber die Konvergenzfragen der Mondtheorie},
\newblock Math. Z, 30(1929), no. 1, 211-227.

\bibitem{Wintner2}
A. Wintner,
\newblock {\it The Analytical Foundations of Celestial Mechanics},
\newblock Princeton University Press, (1941).


\end{thebibliography}
\end{document}